\documentclass[12pt]{amsart}
\usepackage{amsmath,amssymb,amscd,verbatim,xspace,amsthm,tensor}
\usepackage{graphicx,enumerate,tikz,tikz-cd,appendix}
\usepackage{latexsym, epsfig, color}
\usepackage[hidelinks]{hyperref}
\usepackage{fullpage}
\usepackage{pdfsync}
\usepackage[OT2,T1]{fontenc}
\usepackage{amsfonts}
\usepackage{mathrsfs}
\numberwithin{equation}{section}

\theoremstyle{plain}
\newtheorem{theorem}{Theorem}[section]
\newtheorem{proposition}[theorem]{Proposition}

\newtheorem{lemma}[theorem]{Lemma}
\newtheorem{conjecture}[theorem]{Conjecture}

\theoremstyle{definition}
\newtheorem{definition}[theorem]{Definition}
\newtheorem{notation}[theorem]{Notation}

\newtheorem{example}[theorem]{Example}

\theoremstyle{remark}
\newtheorem*{remark}{Remark}

\newcommand\cut[1]{}
\newcommand\further[1]{}
\begin{document}

\title[Ordering of Number Fields and Distribution of Class Groups]{Ordering of Number Fields and Distribution of Class Groups: Abelian Extensions}
\date{}
\author{Weitong Wang}
\address{Shing-Tung Yau Center of Southeast University, 15th floor Yifu Architecture Building, 2 Sipailou, Nanjing, Jiangsu Province 210096 China}
\email{wangweitong@seu.edu}	

\begin{abstract}
	When $p$ divides the ordering of Galois group, the distribution of $\mathbb{Z}_p\otimes\operatorname{Cl}_K$ is closely related to the problem of counting fields with certain specifications.
	Moreover, different orderings of number fields affect the answers of such questions in a nontrivial way.
	So, in this paper, we set up an invariant of number fields with parameters, and consider field counting problems with specifications while the parameters change as a variable.
	The case of abelian extensions shows that the result of counting abelian fields has a main term with parameters.
	The estimate of counting cubic fields with a parameter shows that infinite moment is true for some ordering but not very likely for the others.
\end{abstract}




\maketitle

\section{Introduction}\label{section:intro}

For an extension $K/k$ of number fields, we denote its Galois closure by $\hat{K}$.
Let's first present the following conjecture, which is one of the main statements of \cite{Wang2022DistributionOT}.
Note that the statement here differs by the original one a little on the ordering of number fields for the purpose of applications in this paper.
\begin{definition}\label{def:indicator of (Omega,r)-fields}
	Let $1\leq G\leq S_n$ be a transitive permutation group, and let $k$ be a number field.
	Let $\mathcal{S}:=\mathcal{S}(G)$ be the set of all number fields $(K/k,\psi)$ such that its Galois closure $(\hat{K}/k,\psi)$ is a $G$-extension (see Definition~\ref{def:Gamma-fields}), and that $K=\hat{K}^{H}$ where $H=\operatorname{Stab}_G(1)$ is the stabilizer of $1$.
	Suppose that $\Omega$ is a (nonempty) subset of $G$ that is closed under invertible powering,
	i.e., if $g^a=h$, $h^b=g$, then $g\in\Omega$ if and only if $h\in\Omega$.
	Define for the set $\Omega$, and for all $\gamma=1,2,\dots$,
	\[\mathbf{1}_{(\Omega,\gamma)}(K):=\left\{\begin{aligned}
		&1\quad&\text{if there are exactly }\gamma\text{ rational primes }p\nmid\lvert G\rvert\\
		&&\text{s.t. }I(p)\cap\Omega\neq\emptyset;\\
		&0\quad&\textnormal{otherwise.}
	\end{aligned}\right.\]
	where $I(p)$ here means the inertia subgroup of $p$.
	When $\gamma=0$, define $\mathbf{1}_{(\Omega,0)}(K)=1$ if for any rational prime \(p\) we have \(I(p)\cap\Omega=\emptyset\), and $\mathbf{1}_{(\Omega,0)}(K)=0$ otherwise.
\end{definition}
\begin{conjecture}\label{conj:main body}
	Keep $G,k,\mathcal{S}$ as above.
	Let $\operatorname{id}\notin\Omega$ be a (nonempty) subset of $G$ that is closed under invertible powering, and let $\Theta$ be an invariant of number fields.
	\begin{enumerate}
		\item\label{conj:r rami prime less than r' rami prime} For all $\gamma=0,1,2,\dots$, there exists some $\gamma'$, such that
		\begin{equation*}
			\sum_{\substack{
					K\in\mathcal{S}\\
					\Theta(K)<x
			}}\mathbf{1}_{(\Omega,\gamma)}(K)=o\left(\sum_{\substack{
					K\in\mathcal{S}\\
					\Theta(K)<x
			}}\mathbf{1}_{(\Omega,\gamma')}(K)\right),
		\end{equation*}
		In this case we say that the conjecture~\ref{conj:r rami prime less than r' rami prime} holds for $(({\mathcal{S}},\vartheta),\Omega)$.
		\item\label{conj:r rami prime less than count fields} For all $\gamma=0,1,2,\dots$,
		\begin{equation*}
			\sum_{\substack{
					K\in\mathcal{S}\\
					\Theta(K)<x
			}}\mathbf{1}_{(\Omega,\gamma)}(K)=o\left(\sum_{\substack{
					K\in\mathcal{S}\\
					\Theta(K)<x
			}}1\right)
		\end{equation*}
		In this case we say that the conjecture~\ref{conj:r rami prime less than count fields} holds for the pair $({\mathcal{S}}(\Theta),\Omega)$.
	\end{enumerate}
\end{conjecture}
It is worth mentioning that whenever the conjecture holds, we can prove some results on the distribution of class groups.
To be precise, we have the following.
\begin{definition}\label{def:non-random primes}
	Let $1\leq G\leq S_n$ be a finite transitive permutation group.
	Let $g\in G$ be any permutation.
	Define $e(g):=\gcd(\lvert\langle g\rangle\cdot1\rvert,\dots,\lvert\langle g\rangle\cdot n\rvert)$, i.e., the greatest common divisor of the size of orbits.
	We call $q$ a \emph{non-random prime} for $G$ if $q\mid e(g)$ for some $g\in G$.
	On the other hand, for a permutation $g\in G$, if $q^l\|e(g)$, then we call $g$ an element inertia of type $q^l$.
	Define $\Omega(G,q^l)$ to be the subset of $G$ consisting of all elements inertia of type $q^l$.
	We denote $\bigcup_{l=1}^\infty\Omega(G,q^l)$ by $\Omega(G,q^\infty)$.
\end{definition}
\begin{theorem}\label{thm:statistical results for relative class groups}\cite{Wang2022DistributionOT}
	Let $1\leq G\leq S_n$ be a transitive permutation group, and let $k$ be a number field.
	Let $\mathcal{S}$ be the set of all number fields $(K/k,\psi)$ such that its Galois closure $(\hat{K}/k,\psi)$ is a $G$-extension, and that $K=\hat{K}^{G_1}$.
	Let $H\subseteq G$ be a subgroup such that $\hat{K}^H\subseteq K$ for $K\in\mathcal{S}$.
	If $q$ is a non-random prime for $G$ such that $q\mid[K:\hat{K}^H]$, and Conjecture~\ref{conj:main body}(\ref{conj:r rami prime less than count fields}) holds for $((\mathcal{S},\vartheta),\Omega)$, where $\Omega:=\Omega(G,q^\infty)$, then
	\[\mathbb{P}(\operatorname{rk}_q\operatorname{Cl}(K/\hat{K}^H)\leq r)=0\quad\text{and}\quad\mathbb{E}(\lvert\operatorname{Hom}(\operatorname{Cl}(K/\hat{K}^H),C_q)\rvert)=+\infty,\]
	where $K$ runs over fields in $\mathcal{S}$ for the product of ramified primes in $K/\mathbb{Q}$, and $\operatorname{Cl}(K/\hat{K}^H)$ denotes the relative class group.
\end{theorem}
In this paper, we will follow the idea of the conjecture, and discuss the relation between ordering of number fields and field counting.
We use an example to explain the main result of this paper.
For general case, see Theorem~\ref{thm:counting abelian fields}.
\begin{proposition}
	Let \(G=C_2\times C_2\) with generators \(g_1,g_2\), and let \(\mathcal{S}:=\mathcal{S}(G)\) be the set of \(G\)-fields.
	For each \(K\in\mathcal{S}\), let \(e_1\) be the product of ramified primes such that \(g_1\in I(p)\), where \(I(p)\) is the inertia subgroup of \(p\).
	For \(x>0\), define \(\Theta_x(K):=e_1e_2^{x}\) where \(e_2=P(K)/e_1\) where \(P(K)\) is the product of ramified primes in \(K/\mathbb{Q}\).
	Let \(\Omega=\{g_2,g_1g_2\}\).
	When \(\gamma\geq1\), we have
	\begin{equation*}
		\sum_{\substack{
				K\in\mathcal{S}\\
				\Theta_x(K)<t
			}}\mathbf{1}_{(\Omega,\gamma)}(K)\sim\left\{\begin{aligned}
			&c_1(t)t\quad&\text{ if }x>1\\
			&c_2t(\log\log t)^{\gamma}\quad&\text{ if }x=1\\
			&c_3(x)\frac{t^{x^{-1}}}{\log t}(\log\log t)^{\gamma}\quad&\text{ if }x<1
		\end{aligned}\right.
	\end{equation*}
	where \(c_2>0\) is a constant, and \(c_1\), resp. \(c_3\), is a continuous function defined in \((1,+\infty)\), resp. in \((0,1)\).
	Moreover, the Conjecture~\ref{conj:main body}(\ref{conj:r rami prime less than count fields}) holds for \((\mathcal{S}(\Theta_x),\Omega)\) if and only if \(x\leq1\).
\end{proposition}
In the case of non-abelian extensions, we could show the following.
\begin{theorem}
	Let \(\mathcal{S}:=\mathcal{S}(S_3)\), and define for each \(K_3\in\mathcal{S}\) and \(x>0\) the following invariant
	\begin{equation*}
		\Theta_x(K):=\operatorname{Disc}(K_2/\mathbb{Q})\cdot\operatorname{Disc}(K_6/K_2)^{x/4},
	\end{equation*}
	where \(K_6\) is the Galois closure of \(K_3\) and \(K_2\) is the unique quadratic subfield of \(K_6\).
	If \(0<x<1\), then for an integer \(\gamma>0\), we have as \(t\to\infty\) that 
	\begin{equation*}
		\sum_{\substack{
				K_3\in\mathcal{S}\\
				\Theta_x(K_3)<t
		}}\mathbf{1}_{(\Omega,\gamma)}(K_3)\asymp\frac{t^{x^{-1}}}{\log t}(\log\log t)^{\gamma},
	\end{equation*}
	and in this case Conjecture~\ref{conj:main body}(\ref{conj:r rami prime less than r' rami prime}) holds for \((\mathcal{S}(\Theta_x),\Omega)\).
\end{theorem}
On the other hand, the following example explains that the conjecture is false if we count number fields by discriminant.
\begin{example}[Ordering by discriminant]
	According to the work of Davenport, Heilbronn~\cite[Theorem 3]{Davenport1971Cubic}, and Bhargava, Shankar, Tsimerman~\cite[Theorem 8]{bhargava2013davenport}, we know that counting nowhere totally ramified degree $3$ cubic fields will give the main term $cx$ where $c$ is a nonzero constant, which is the same main term as counting all cubic fields by discriminant.
	This already contradicts the analogous statement of Conjecture~\ref{conj:main body}(\ref{conj:r rami prime less than count fields}) when ordering fields by discriminant, the weaker one.
	So we can conclude that when cubic fields ordered by discriminant with non-random prime $3$ is a counter-example for the conjecture.
\end{example}

\section*{Acknowledgement}
  The author would like to thank Yuan Liu for many useful conversations.
  This work was done with the support of Jiangsu Funding Program for Excellent Postdoctoral Talent.

\section{Basic notations}\label{section:notation}

In this section we introduce some of the notations that will be used in the paper.
We use some standard notations coming from analytic number theory.
For example, write a complex number as $s=\sigma+it$.
The notation $q^l\|n$ means that $q^l\mid n$ but $q^{l+1}\nmid n$.
Denote the Euler function by $\phi(n)$.
Let $\omega(n)$ counts the number of distinct prime divisors of $n$ and so on.

We also follow the notations of inequalities with unspecified constants from Iwaniec and Kowalski~\cite[Introduction, p.7]{iwaniec2004analytic}.
Let's just write down the ones that are important for us.
Let $X$ be some space, and let $f,g$ be two functions.
Then $f(x)\ll g(x)$ for $x\in X$ means that $\lvert f(x)\rvert\leq Cg(x)$ for some constant $C\geq0$.
Any value of $C$ for which this holds is called an implied constant.
We use $f(x)\asymp g(x)$ for $x\in X$ if $f(x)\ll g(x)$ and $g(x)\ll f(x)$ both hold with possibly different implied constants.
We say that $f=o(g)$ as $x\to x_0$ if for any $\epsilon>0$ there exists some (unspecified) neighbourhood $U_\epsilon$ of $x_0$ such that $\lvert f(x)\rvert\leq\epsilon g(x)$ for $x\in U_\epsilon$.

For the convenience of stating the results, we make the following notation.
\begin{definition}\label{def:set of primes}
	Let $\mathcal{P}$ be the set of rational prime numbers, and for each pair $(m,n)$ of coprime numbers, we define
	  \[\mathcal{P}(n):=\{p\in\mathcal{P}\mid\gcd(p,n)=1\}\text{ and }\mathcal{P}(m,n):=\{p\in\mathcal{P}\mid p\equiv m\bmod{n}\}.\]
\end{definition} 
Since there are multiple ways to describe field extensions, we give the following two definitions to make the term like ``the set of all non-Galois cubic number fields'' precise. 
\begin{definition}\label{def:Gamma-fields}
	For a field $k$, by a \emph{$\Gamma$-extension of $k$}, we mean an isomorphism class of pairs $(K,\psi)$, where $K$ is a Galois extension of $k$, and $\psi:\operatorname{Gal}(K/k)\cong\Gamma$ is an isomorphism.  
	An isomorphism of pairs $(\tau,m_\tau):(K,\psi)\to(K',\psi')$ is an isomorphism $\tau: K\to K'$ such that the map $m_\tau: \operatorname{Gal}(K/k)\to\operatorname{Gal}(K'/k)$ sending $g$ to $\tau\circ g\circ\tau^{-1}$ satisfies $\psi'\circ m_\tau=\psi$.  
	We sometimes leave the $\psi$ implicit, but this is always what we mean by a $\Gamma$-extension.  
	We also call $\Gamma$-extensions of $\mathbb{Q}$ \emph{$\Gamma$-fields.}
\end{definition}
\begin{definition}\label{def:set of fields}
	Let $G\subseteq S_n$ whose action on $\{1,2,\dots,n\}$ is transitive.
	Let $k$ be a number field.
	Let ${\mathcal{S}}(G;k)$ be the isomorphic classes of pairs $(K,\psi)$ such that the Galois closure $(\hat{K},\psi)$ is a $G$-field and that $K=\hat{K}^{H}$, where $H=\operatorname{Stab}_G(1)$ is the stabilizer of $1$.
	In other words $\psi$ defines the Galois action of $G$ on $K$ over $k$ and $[K:k]=n$.
	If the base field $k=\mathbb{Q}$, then we just omit it and write ${\mathcal{S}}(G):={\mathcal{S}}(G,\mathbb{Q})$.
\end{definition}
Then, we give the notation of counting number fields.
\begin{definition}\label{def:counting number fields}
	Let $\mathcal{S}=\mathcal{S}(G;k)$ where $G$ is a transitive permutation group, 
	and let $\Theta:\mathcal{S}\to\mathbb{R}^+$ be an invariant of number fields (e.g. discriminant or product of ramified primes).
	Write \(\mathcal{S}(\Theta)\) to be the set of fields \(\mathcal{S}(G;k)\) with ordering given by \(\Theta\).
	Define
	\[N(\mathcal{S}(\Theta);x):=\sum_{K\in\mathcal{S},\Theta(K)<x}1.\]
	If there is no danger of confusion, we abbreviate it as $N(x)$.
	If $f$ is a function defined over $\mathcal{S}$, then we define
	\[N(\mathcal{S}(\Theta);f;x):=\sum_{K\in\mathcal{S},\Theta(K)<x}f(K),\]
	which could be abbreviated as \(N_f(x)\).
	In particular, if $f=\mathbf{1}_{(\Omega,\gamma)}$ (see Definition~\ref{def:indicator of (Omega,r)-fields}), then just write
	\[N(\mathcal{S}(\Theta);(\Omega,\gamma);x):=N(\mathcal{S}(\Theta);\mathbf{1}_{(\Omega,\gamma)};x),\]
	which could be abbreviated as \(N_{\gamma}(x)\).
\end{definition}

\section{Invariant of number fields}\label{section:invariant}

We first discuss a little on the idea of setting up an invariant on number fields.
The following discussion mainly comes from Wood~\cite[Section 2.1]{wood2010probabilities}.
Let $G$ be a transitive permutation group in $S_n$.
\begin{definition}\label{def:equivalent under invertible powering and conjugation}
	Let $G$ be a finite transitive group.
	Given two elements $g_1,g_2\in G$, if there exists some $a,b\in\mathbb{Z}$ and some $h_1,h_2\in G$ such that
	\[g_1^a=h_2g_2h_2^{-1}\quad\text{and}\quad g_2^b=h_1g_1h_1^{-1},\]
	then we say that $g_1$ and $g_2$ are equivalent under closed powering and conjugation, denoted by $g_1\sim g_2$.
	If there is no danger of confusion, we simply say that $g_1$ and $g_2$ are equivalent.
\end{definition}
We will define an invariant on the set \(\mathcal{S}:=\mathcal{S}(G;k)\) of fields by the following steps.
First we need to assign a number for each element of \(G\).
\begin{definition}\label{def:functions on G}
	Let \(\nu:G\to\mathbb{R}_{\geq0}\) be a function.
	We call \(\nu\) a weight on \(G\) if it satisfies the following conditions:
	\begin{enumerate}
		\item \(\nu(g)=0\) if and only if \(g=\operatorname{id}_G\);
		\item \(\nu(g)=\nu(h)\) if \(g\sim h\).
	\end{enumerate}
\end{definition}
Then we can define an invariant using the weight \(\nu\).
\begin{definition}\label{def:invariant of number fields}
	Let \(\mathcal{S}_{\mathfrak{p}}:=\mathcal{S}_{\mathfrak{p}}(G)\) be the isometry classes of {\'e}tale \(k_{\mathfrak{p}}\)-algebras with \(G\)-action, where \(\mathfrak{p}\) is a finite or infinite prime of \(k\).
	For each prime \(\mathfrak{p}\mid\lvert G\rvert\infty\), let \(\vartheta:\mathcal{S}_\mathfrak{p}\to\mathbb{R}_{\geq0}\) be any function.
	Let \(\vartheta:\mathcal{S}_\mathfrak{p}\to\mathbb{R}_{\geq0}\) be a function defined as follows:
	\begin{equation*}
		\vartheta(\Sigma_\mathfrak{p}):=\left\{\begin{aligned}
			&\nu(y_\mathfrak{p})\quad&\text{if }\mathfrak{p}\nmid\lvert G\rvert\infty,\text{ where }\Sigma_\mathfrak{p}=\operatorname{Ind}^G_H M,\text{ and }M/k_\mathfrak{p}\text{ a field extension,}\\
			&&\text{ and }y_\mathfrak{p}\text{ is a generator of tame inertia in }H:=\operatorname{Gal}(M/k_\mathfrak{p})\subseteq G;\\
			&\vartheta(\Sigma_\mathfrak{p})\quad&\text{if }\mathfrak{p}\mid\lvert G\rvert\infty.
		\end{aligned}\right.
	\end{equation*}
	We call such \(\vartheta:\mathcal{S}_\mathfrak{p}\to\mathbb{R}_{\geq0}\) a \emph{weight function} on the isometry classes \(\mathcal{S}_\mathfrak{p}\) dependent on \(\nu\).
	Let \(\hat{\mathcal{S}}\) be the isomorphic classes of Galois \(G\)-fields, and let \(\Theta:\hat{\mathcal{S}}\to\mathbb{R}_{\geq0}\) be an invariant defined by the formula \begin{equation*}
		\Theta(L):=\prod_{\mathfrak{p}} (N\mathfrak{p})^{\vartheta(L_\mathfrak{p})},
	\end{equation*} 
	where \(N\mathfrak{p}=\operatorname{Nm}_{k/\mathbb{Q}}\mathfrak{p}\) if \(p\nmid\infty\) and \(N\mathfrak{p}=1\) if \(p\mid\infty\).
	For each \(K\in\mathcal{S}(G)\), define \(\Theta(K):=\Theta(\hat{K})\), where \(\hat{K}\) is the Galois closure of \(K\).
	We call the invariant \(\Theta\) a \emph{counting function} determined by \(\nu\) and \(\vartheta\), and we sometimes leave \(\nu\) and \(\vartheta\) implicit if there is no danger of confusion.
\end{definition}
Note that \(\hat{K}\in\hat{\mathcal{S}}\) according to Definition~\ref{def:set of fields}.

Assume for now that \(G\) is abelian and \(k=\mathbb{Q}\).
Let \(\nu\) be a weight function on \(G\).
For each prime \(p\), let \(\vartheta\) be a counting function on the isometry classes \(\mathcal{S}_p\) of {\'e}tale \(\mathbb{Q}_p\)-algebras with \(G\)-action.
If \(p\nmid\infty\), then denote by \(\zeta_p\) a generator of \(\mu(\mathbb{Q}_p)\), where \(\mu(\mathbb{Q}_p)\) is the group of unity of \(\mathbb{Q}_p\).
If \(\rho:\mathbb{Q}_p^*\to G\) is a local homomorphism with image \(H\), then \(\rho\) corresponds to a local field extension \(K_{\mathfrak{p}}/\mathbb{Q}_p\) with Galois group \(H\).
Let \(K_p:=\operatorname{Ind}^G_H K_{\mathfrak{p}}\) be the {\'e}tale \(\mathbb{Q}_p\)-algebra with the Galois action from \(G\).
According to Definition~\ref{def:invariant of number fields}, we see that \(\vartheta(K_p)=\nu(\rho(\zeta_p))\) if \(p\nmid\lvert G\rvert\infty\).
Because the inertia subgroup is given by \(\rho(\mathbb{Z}_p^*)\), and the map \(\rho:\mathbb{Z}_p^*\to G\) factors through \(\mu(\mathbb{Q}_p)\) in this case.
So, actually the computation of \(\vartheta(K_p)\) reduces to checking the image of \(\zeta_p\) under \(\rho\).
\begin{definition}\label{def:invariant of abelian fields}
	Let \(p\nmid\lvert G\rvert\infty\), and let \(\rho:\mathbb{Z}_p^*\to G\) be a homomorphism.
	Define
	\begin{equation*}
		\vartheta(\rho):=\nu(\rho(\zeta_p)).
	\end{equation*}
\end{definition}
For example, the product of ramified primes comes from the weight function $\nu(g)=1$ for all \(g\in G\), up to wildly ramified primes.
On the other hand, the discriminant of $K\in\mathcal{S}$ requires the formula of computing $\operatorname{Disc}(\rho_p)$ for each $p$.
But it is also true that for tamely ramified primes, $\operatorname{Disc}(\rho_p)$ is determined by $I(p)$, so it is enough to assign the correct weight to each element \(g\in G\).
\begin{example}
	To give a concrete example, let $G\cong C_6$ with generator $g$.
	The weight function \(\nu(g)=1\) clearly gives the product of ramified primes, up to wildly ramified primes.
	For a tamely ramified prime $p$, we have the following computation.
	Let \(\rho\) be a local homomorphism.
	\begin{enumerate}
		\item if $\rho(\zeta_p)\sim g$, then $\operatorname{Disc}(\rho_p)=p^5$;
		\item if $\rho(\zeta_p)\sim g^2$, then $\operatorname{Disc}(\rho_p)=p^4$;
		\item if $\rho(\zeta_p)\sim g^3$, then $\operatorname{Disc}(\rho_p)=p^3$.
	\end{enumerate}	
	So, let \(\nu\) be the weight given by the formula \(g\mapsto 5,g^2\mapsto 4,g^3\mapsto 3\), and we will obtain the discriminant, up to wildly ramified primes.
\end{example}
Then let's show the computation of $\mathbf{1}_{(\Omega,\gamma)}$ in the view of Class Field Theory.
\begin{definition}\label{def:indicator of inertia for morphisms}    
	Let $\operatorname{id}\notin\Omega$ be a subset of $G$ closed under invertible powering, and let \(p\nmid\infty\) be a prime of \(\mathbb{Q}\).
	If \(\rho:\mathbb{Z}_p^*\to G\) is a local homomorphism, then define 
	\[\mathbf{1}_{\Omega}(\rho)=\left\{\begin{aligned}
		&1\quad\text{if }p\nmid\lvert G\rvert\text{ and }\langle\rho(\zeta_p)\rangle\cap\Omega\neq\emptyset;\\
		&0\quad\text{otherwise.}
	\end{aligned}\right.\]
	For each positive integer \(\gamma\), and for each homomorphism \(\chi:\prod_{p\nmid\infty}\mathbb{Z}_p^*\to G\), let
	\begin{equation*}
		\mathbf{1}_{(\Omega,\gamma)}(\chi):=\left\{\begin{aligned}
			1\quad&\text{ if there are }\gamma\text{ primes }p\\
			&\text{such that }p\nmid\lvert G\rvert\text{ and }\mathbf{1}_{\Omega}(\chi\vert_p)=1\\
			0\quad&\text{otherwise.}
		\end{aligned}\right.
	\end{equation*}
\end{definition}
One can check that if \(\tilde{\rho}:\operatorname{C}_{\mathbb{Q}}\to G\) is an Artin reciprocity map corresponding to the class field \(K\), and $\rho:\prod_{p\nmid\infty}\mathbb{Z}_p^*\to G$ is a homomorphism defined by \(\rho=\prod_{p\nmid\infty}\tilde{\rho}\vert_p\), then $\mathbf{1}_{(\Omega,\gamma)}(K)=\mathbf{1}_{(\Omega,\gamma)}(\rho)$.

In the examples of non-abelian fields we are going to consider, there are shorter ways to give flexible enough counting functions of number fields.
\begin{definition}\label{def:invariant general}
Let $G=H_0\supsetneq H_1\supsetneq\cdots\supsetneq H_l=\{\operatorname{id}\}$ be a chain of subgroups of $G$.
Given any $K\in\mathcal{S}(G)$ with is Galois closure $\tilde{K}$, we then obtain a chain of field extensions using the notation $K_i=\tilde{K}^{H_i}$, i.e.,
\[\mathbb{Q}=K_0\subsetneq K_1\subsetneq\cdots\subsetneq K_l=\tilde{K}.\]
Then for any $(x_1,\dots,x_l)\in\mathbb{R}_+^l$, we just define $\Theta_{(x_1,\dots,x_l)}(K):=e_1^{x_1}\cdots e_l^{x_l}$ where $e_i=\operatorname{Nm}_{K_{i-1}/\mathbb{Q}}\Delta(K_i/K_{i-1})$ where $\Delta$ denotes the ideal of discriminant.
\end{definition}
If we ignore the wildly ramified primes, then discriminant could be realized by Definition~\ref{def:invariant of number fields}, so this notation \(\Theta_{\underline{x}}\) is indeed a counting function.
\begin{example}
Let $G=S_3$, i.e., $K\in\mathcal{S}(S_3)$ means that $K$ is a non-Galois cubic number field.
Consider the chain $S_3=H_0\supset\langle(123)\rangle\supset1$.
In terms of field extensions, we just write it as
\[\mathbb{Q}\subset K_2\subset K_6\]
where the index means the degree of extensions with respect to $\mathbb{Q}$.
One can check that $\vartheta_{(1,1/4)}$ is the product of ramified prime if we ignore wildly ramified primes.
Similarly, $\vartheta_{(1,1/2)}$ is just the usual (absolute) discriminant of $K/\mathbb{Q}$.
Moreover, let $\Omega=\{(123),(132)\}\subseteq S_3$.
For each $K\in\mathcal{S}(S_3)$, we have
\[\mathbf{1}_{(\Omega,\gamma)}(K)=1\]
if and only if $\dfrac{\Theta_{(1,1/2)}}{\Theta_{(1,1/4)}}$ has exactly $\gamma$ prime factors other than $2,3$.
\end{example}
Before we get into any concrete examples of counting fields, let's prove some general results here.
Provided that $\Theta$ is a counting function of $\mathcal{S}:=\mathcal{S}(G)$, dependent on the weight function \(\nu:G\to\mathbb{R}_{\geq0}\) and the counting function \(\vartheta\) on the isometry classes \(\mathcal{S}_p\), where $G$ is a transitive permutation group, then clearly $\Theta^a$, where $a>0$ is any real number, is also a counting function dependent on \(\nu^a\) and \(\vartheta^a\).
\begin{lemma}\label{lemma:projective space?}
	Let $G$ be a transitive permutation group, let $\Omega$ be a subset of $G$, and let $\Theta$ be a counting function of $\mathcal{S}:=\mathcal{S}(G)$.
	If $f,g:\mathcal{S}\to\mathbb{R}$ are two functions with some constant $C\in\mathbb{R}$ or $C=\pm\infty$, such that
	\begin{equation}\label{eqn:projective space?}
		\lim_{X\to\infty}\frac{N(\mathcal{S}(\Theta);f;X)}{N(\mathcal{S},\Theta;g;X)}=C,
	\end{equation}
	then for any $a>0$, we also have
	\[\lim_{X\to\infty}\frac{N(\mathcal{S}(\Theta^a);f;X)}{N(\mathcal{S},\Theta^a;g;X)}=C.\]
\end{lemma}
\begin{proof}
	We prove the case when $C$ is a constant, and the other cases are left to the reader, for they are similar to each other.
	Let $\epsilon>0$ be any positive number.
	Then by our condition (\ref{eqn:projective space?}), which says that the limit exists, there exists some $N>0$ such that for all $X>N$ we have
	\[\left\lvert\frac{N(\mathcal{S}(\Theta);f;X)}{N(\mathcal{S}(\Theta);g;X)}-C\right\rvert<\epsilon.\]
	Note that for any $X>0$, we have $\vartheta<X$ if and only if $\vartheta^a<X^a$.
	This implies that for all $X>N^{a}$, we also have
	\[\left\lvert\frac{N(\mathcal{S}(\Theta^a);f;X)}{N(\mathcal{S}(\Theta^a);g;X)}-C\right\rvert<\epsilon.\]
	And we are done for the proof when $C$ is a constant.
\end{proof}
In the sense of the Conjecture~\ref{conj:main body}, we can see that if the conjecture holds for $(\mathcal{S}(\Theta),\Omega)$, then it will also hold for $(\mathcal{S}(\Theta^a),\Omega)$.
Inspired by the conjecture and the above lemma, we make the following definition.
\begin{definition}\label{def:function R}
	Let $G$ be a transitive permutation group with a subset $\Omega$ closed under invertible powering and conjugation.
	Let $\mathcal{S}:=\mathcal{S}(G)$, and $\Theta$ be a counting function with parameters $\underline{x}=(x_1,\dots,x_{l})\in\mathbb{R}_+^{l}$, define for each $(\gamma_1,\gamma_2)\in\mathbb{N}^2$ the function $R^{\gamma_1}_{\gamma_2}:\mathbb{R}_+^{l}\to\mathbb{R}_{\geq0}$, given by the expression
	\[R^{\gamma_1}_{\gamma_2}(\underline{x}):=\lim_{X\to\infty}\frac{N(\mathcal{S}(\Theta_{\underline{x}});(\Omega,\gamma_1);X)}{N(\mathcal{S}(\Theta_{\underline{x}});(\Omega,\gamma_2);X)}.\]
\end{definition}
Note that $R$ may not be well-defined everywhere on $\mathbb{R}_+^{l}$.
By Lemma~\ref{lemma:projective space?}, we see that if $R^{\gamma_1}_{\gamma_2}(x_1,\dots,x_{l})=c$ for some real number $c\geq0$, then so is $R^{\gamma_1}_{\gamma_2}(x_1^a,\dots,x_l^a)$ for all $a>0$.
\begin{theorem}\label{thm:function R at infinity}
	Let $G$ be a transitive permutation group with a chain of subgroups $G=H_0\supsetneq H_1\supsetneq\cdots\supsetneq H_l=\{\operatorname{id}\}$, and let $\Omega:=H_{l-1}-H_l$.
	Let $\mathcal{S}:=\mathcal{S}(G)$, and $\Theta_{\underline{x}}$ be an invariant with parameter $\underline{x}\in\mathbb{R}^l_+$ as in Definition~\ref{def:invariant general}.
	Define $\Theta_x:=\Theta(c_1,\dots,c_{l-1},x)$, where $c_1,\dots,c_{l-1}$ are positive real numbers and define
	\[R^0(x):=\lim_{t\to\infty}\frac{N(\mathcal{S}(\Theta_x);(\Omega,0);t)}{N(\mathcal{S}(\Theta_x);t)}.\]
	If $R^0(x)$ is a well-defined function in $[x_0,\infty)$ with $R^0(x_0)>0$, where $x_0>c_i$ is a real number, for all $i=1,\dots,l-1$, and if $N(\mathcal{S}(\Theta_{x_0});t)\asymp f(t)$, where $f$ is a positive real function such that for each \(a>0\) we have
	\[\frac{f(a)}{f(ab)}=o(1),\]
	 as \(b\to\infty\) where the implied constant is independent of \(a\), then \(R^0(x)\) is continuous when \(x>x_0\) and 
	\[\lim_{x\to\infty}R^0(x)=1.\]
\end{theorem}
An example of such \(f\) is \(x^{\alpha}(\log x)^{\beta}(\log\log x)^{\gamma}\), where \(\alpha>0,\beta,\gamma\) are all real numbers.
\begin{proof}
	Let
	\begin{equation*}
		N_x(t):=N(\mathcal{S},\Theta_x;t)\quad\text{and}\quad N_{x,0}(t):=N(\mathcal{S}(\Theta_x);(\Omega,0);t).
	\end{equation*}
	Using our notation in Definition~\ref{def:invariant general}, a field extension $\tilde{K}/K_{l-1}$ is unramified, if and only if
	\[\mathbf{1}_{(\Omega,0)}(K)=1.\]
	According to Definition~\ref{def:counting number fields} and Definition~\ref{def:indicator of (Omega,r)-fields}, it is clear that
	\[\frac{N_{x,0}(t)}{N_{x}(t)}\leq1,\]
	i.e., $R(x)\leq1$ for each $x>x_0$.
	For a fixed $t>0$, the function $N_x(t)$ is non-increasing with respect to $x$, and the function $N_{x,0}(t)$ is independent of $x$, so the ratio
	\[R^0(x,t):=\frac{N_{x,0}(t)}{N_x(t)}\]
	is non-decreasing with respect to $x$.
	In particular, when $x$ is large enough, $R^0(x,t)=1$ because the discriminant $\operatorname{Nm}_{K_{l-1}/\mathbb{Q}}\Delta(\tilde{K}/K_{l-1})^x>t$ if it is non-trivial.
	According to the condition $N_{x_0}(t)$, there exists some constant $C_1,C_2>0$ such that
	\[C_1f(t)\leq N_{x_0}(t)\leq C_2f(t)\]
	for large enough $t$.
	Let $\epsilon>0$ be a small real number.
	Then there exists some $X,T>0$ such that for all $x>X$, and for all $t>T$, we have
	\[\begin{aligned}
		&\frac{1}{2}R(x_0)<R^0(x_0,t)<\frac{3}{2}R(x_0)\\
		\text{and }&\frac{2C_2f(\frac{t}{2^x})}{R^0(x_0)C_1f(t)}<\epsilon.
	\end{aligned}\]
	Note that if $\operatorname{Nm}\Delta K_l/K_{l-1}$ is nontrivial, then $\operatorname{Nm}\Delta K_l/K^{l-1}\geq2$.
	For any $(xx_0,t)$ with $x>X$ and $t>T$, we have
	\[\begin{aligned}
		&1-R^0(xx_0,t)\\
		=&1-\frac{N_{x_0}(t)}{N_{x_0,0}(t)+\#\{K\in\mathcal{S}\mid\Theta_{xx_0}(K)<x\text{ and }1<\operatorname{Nm}\Delta(K_l/K_{l-1})^{xx_0}<x\}}\\
		\leq&\frac{\#\{K\in\mathcal{S}\mid\Theta_{xx_0}(K)<x\text{ and }1<\operatorname{Nm}\Delta(K_l/K_{l-1})^{xx_0}<x\}}{N_{x_0}(t)}\\
		\leq&\frac{N_{x_0}(t/2^x)}{N_{x_0,0}(t)}=\frac{N_{x_0}(t/2^x)}{R^0(x_0,t)N_{x_0}(t)}\\
		\leq&\frac{C_2f(\frac{t}{2^x})}{\frac{1}{2}R^0(x_0)C_1f(t)}<\epsilon.
	\end{aligned}\]
	This shows that
	\[\lim_{(t,x)\to\infty}R^0(x,t)=1.\]
	We've shown that for all large enough $t$, we have
	\[\lim_{x\to\infty}R^0(x,t)=1,\]
	so the iterated limit exists, and for \(x\geq x_0\), the limit \(R^0(x)=\lim_{t\to\infty}R^0(x,t)\) is well-defined.
	This implies that the iterated limit exists, and
	\[\lim_{x\to\infty}\lim_{t\to\infty}R^0(x,t)=\lim_{x\to\infty}R^0(x)=1.\]
\end{proof}
The relation of field counting and distribution of class groups could be described as follows.
For each \(\underline{x}\in\mathbb{R}_+^l\) and for each \(\gamma=0,1,2,\dots\), define
\begin{equation*}
	R^\gamma(\underline{x}):=\lim_{t\to\infty}\frac{N(\mathcal{S}(\Theta_{\underline{x}});(\Omega,\gamma);t)}{N(\mathcal{S}(\Theta_{\underline{x}});t)}.
\end{equation*}
\begin{proposition}
	Let \(\underline{x}\in\mathbb{R}_+^l\), and assume that \(R^\gamma(\underline{x})\) is a well-defined function for all \(\gamma=0,1,2,\dots\).
	If there exists \(M\geq0\) such that \(R^{\gamma}(\underline{x})=0\) for all \(\gamma\geq M\) and such that
	\begin{equation*}
		\sum_{\gamma<M}R^{\gamma}(\underline{x})=c<1,
	\end{equation*}
	then for each non-random prime \(q\) of \(G\), we have
	\begin{equation*}
		\mathbb{E}(\lvert\operatorname{Hom}(\operatorname{Cl}_K,C_q)\rvert)=+\infty.
	\end{equation*}
\end{proposition}
\begin{proof}
	According to \cite[Theorem 4.5]{Wang2022DistributionOT}, it suffices to show that there exists some constant \(0\leq C<1\) such that
	\begin{equation*}
		\mathbb{P}(\operatorname{rk}_q\operatorname{Cl}_K\leq\gamma)<C.
	\end{equation*}
	According to \cite[Theorem 1]{RZ1969ClassRank}, we know that \(\operatorname{rk}_q\operatorname{Cl}_K\geq\#\{p\mid e_K(p)\equiv0\bmod{q}\}-2(n-1)\), where \(n=[K:\mathbb{Q}]\).
	So we see that
	\begin{equation*}
		\begin{aligned}
			\mathbb{P}(\operatorname{rk}_q\operatorname{Cl}_K\leq\gamma)=&\lim_{t\to\infty}\frac{\#\{K\mid\Theta_{\underline{x}}(K)<t\text{ and }\operatorname{rk}_q\operatorname{Cl}_K\leq\gamma\}}{N(\mathcal{S}(\Theta_{\underline{x}});t)}\\
			\leq&\lim_{t\to\infty}\frac{\sum_{i=0}^{\gamma+2(n-1)}N(\mathcal{S}(\Theta_{\underline{x}});(\Omega,i);t)}{N(\mathcal{S},\Theta_t;t)}\\
			\leq&\lim_{x\to\infty}\frac{\sum_{i=0}^{\gamma+M+2(n-1)}N(\mathcal{S}(\Theta_{\underline{x}});(\Omega,i);t)}{N(\mathcal{S}(\Theta_{\underline{x}});t)}=c<1.
		\end{aligned}
	\end{equation*}
	This shows that the probability \(\mathbb{P}(\operatorname{rk}_q\operatorname{Cl}_K\leq\gamma)\) is globally bounded above by a number \(c<1\), hence infinite \(C_q\)-moment.
\end{proof}

\section{Tauberian Theorem and Dirichlet series}\label{section:tauberian theorem}
In this section, we first construct a generating series generalizing the usual Dirichlet series, and then make a brief introduction to Delange's Tauberian Theorem together with its application towards the series we constructed.
\subsection{Generating series}
We need some notations before constructing the series.
\begin{definition}
	\begin{enumerate}[(i)]
		\item Let $\imath:\mathbb{N}\to\mathbb{R}_{\geq0}$ be an injective increasing map and let $S\subset\mathbb{R}_{\geq0}$ be the image of $\imath$.
		If for each $N>0$, for all but finitely many $n\in\mathbb{N}$ we have $\imath(n)>N$, then we say that $(\imath,S)$ is a set of index.
		When there is no danger of confusion, we just omit $\imath$ and say $S$ is an index.
		\item Let $S$ be a set of index, and let $a:S\to\mathbb{R}_{\geq0}$ be a map.
		Then define 
		\[D_a(s):=\sum_{d\in S}a_dd^{-s},\]
		to be its generating series, i.e., $a_d=a(d)$ for each $d\in S$.
	\end{enumerate} 
\end{definition}
The first example is of course when $S=\mathbb{N}$, and $D_a(s)$ is the usual Dirichlet series with non-negative real coefficients.
Of course, when we want to emphasize that it is a complex function, we can just write $f(s)=D_a(s)$
In this paper, we'll usually let capital letters denote the summatory functions.
For example, let
\[A(x):=\sum_{d<x}a_d.\]
It is also common to obtain a series $f(s)=D_a(s)$ from some other methods, say Euler products, without knowing much about the map $a:S\to\mathbb{R}_{\geq0}$ itself.
Since we are mainly interested in the coefficients, i.e., the arithmetic function $a:S\to\mathbb{R}_{\geq0}$, let's make the following definition so that we can compare the coefficients of generating series.
\begin{definition}\label{def:partial ordering of generating series}
	Let $S$ be an index set.
	Let $D_a(s)=\sum_{d\in S}a_{d}d^{-s}$ and $D_b(s)=\sum_{d\in S}b_{d}d^{-s}$ be two generating series with non-negative real coefficients.
	We say that $D_a(s)\leq D_b(s)$ if $a_{d}\leq b_{d}$ for all $d\in S$.
\end{definition}
Then we establish the basic properties of the generating series.
In the rest of this section, let $(\imath,S)$ be a fixed set of index, and let $a:S\to\mathbb{R}_{\geq0}$ be a map with generating series $D_a(s)$.
The first result shows that, just like Dirichlet series, the generating series $D_a(s)$ converges in a half-plane.
\begin{proposition}\label{prop:uniform covergence of generating series}
	Suppose that $D_a(s)$ converges at a point $s=s_0$, and that $H>0$ is a positive real number.
	Then the series $D_a(s)$ is uniformly convergent in the sector $\mathcal{T}:=\{s:\sigma\geq\sigma_0,\lvert t-t_0\rvert\leq H(\sigma-\sigma_0)\}$.
\end{proposition}
The proof is similar to the case of Dirichlet series.
See \cite[Theorem 1.1]{montgomery2006multiplicative} for example.
\begin{proof}
	Let $R(u):=\sum_{d>u}a_dd^{-s_0}$ be the remainder term of the series $D_a(s_0)$.
	First we show that for any $s$,
	\begin{equation}\label{eqn:uniform covergence of generating series 1}
		\begin{aligned}
			\sum_{d=\imath(M+1)}^{\imath(N)}a_dd^{-s}=&R(\imath(M))(\imath(M))^{s_0-s}-R(\imath(N))(\imath(N))^{s_0-s}\\
			+&(s_0-s)\int_{\imath(M)}^{\imath(N)}R(u)u^{s_0-s-1}\operatorname{d}u.
		\end{aligned}
	\end{equation}
	Note that $a_{\imath(j)}=(R(\imath(j-1))-R(\imath(j)))\imath(j)^{s_0}$.
	According to Riemann-Stieltjes integral and integral by parts, we have
	\[\begin{aligned}
		\sum_{d=\imath(M+1)}^{\imath(N)}a_dd^{-s}&=-\int_{\imath(M)}^{\imath(N)}u^{s_0-s}\operatorname{d}R(u)\\
		&=-u^{s_0-s}R(u)\Big\vert_{\imath(M)}^{\imath(N)}+\int_{\imath(M)}^{\imath(N)}R(u)\operatorname{d}u^{s_0-s}\\
		=&R(\imath(M))(\imath(M))^{s_0-s}-R(\imath(N))(\imath(N))^{s_0-s}\\
		+&(s_0-s)\int_{\imath(M)}^{\imath(N)}R(u)u^{s_0-s-1}\operatorname{d}u.
	\end{aligned}\]
	If $\lvert R(u)\rvert\leq\varepsilon$ for all $u\geq M$ and if $\sigma\geq\sigma_0$, then from (\ref{eqn:uniform covergence of generating series 1}) we see that
	\[\left\lvert\sum_{d=\imath(M+1)}^{\imath(N)}a_dd^{-s}\right\rvert\leq2\varepsilon+\varepsilon\lvert s-s_0\rvert\int_M^\infty u^{\sigma_0-\sigma-1}\operatorname{d}u\leq\left(2+\frac{\lvert s-s_0\rvert}{\sigma-\sigma_0}\right)\varepsilon.\]
	For $s$ in the prescribed region we see that 
	\[\lvert s-s_0\rvert\leq\sigma-\sigma_0+\lvert t-t_0\rvert\leq(H+1)(\sigma-\sigma_0),\]
	so that the sum $\sum_{d=\imath(M+1)}^{\imath(N)}a_dd^{-s}$ is uniformly small, and the result follows from the uniform version of Cauchy principle.
\end{proof}
\begin{remark}
	Proposition~\ref{prop:uniform covergence of generating series} implies that the generating series $D_a(s)$ converges in a half-plane $\sigma\geq\sigma_c$, and we call it \emph{abscissa of convergence}.
	Moreover, since each term of the series is a regular function (complex analytic function or holomorphic function) in the open half-plane $\sigma>0$ and the series itself is locally uniformly convergent in $\sigma>\sigma_c$, we see that series $D_a(s)$ is also a regular function in $\sigma>\sigma_c$ by Weierstrass principle.
\end{remark}
The following result is an analogous statement of Dirichlet series, saying that $D_a(s)$ could be expressed in the form of an integral using the idea of Riemann-Stieltjes integral.
See \cite[Theorem 1.3]{montgomery2006multiplicative} for example.
\begin{proposition}\label{prop:integral form of D(s)}
	Let $A(x):=\sum_{d<x}a_d$ be the summatory function.
	Denote the abscissa of convergence of $D_a(s)$ by $\sigma_c$.
	If $\sigma_c<0$, then $A(x)$ is bounded and 
	\begin{equation}\label{eqn:integral form of D(s) 1}
		D_a(s)=s\int_0^\infty x^{-s-1}A(x)\operatorname{d}x,
	\end{equation}
	for $\sigma>0$.
	If $\sigma_c\geq0$, then
	\begin{equation}\label{eqn:integral form of D(s) 2}
		\limsup_{x\to\infty}\frac{\log\lvert A(x)\rvert}{\log x}=\sigma_c,
	\end{equation}
	and (\ref{eqn:integral form of D(s) 1}) holds for $\sigma>\sigma_c$.
\end{proposition}
\begin{proof}
	Let's do the following computation:
	\[\begin{aligned}
		\sum_{n<N}a_dd^{-s}&=\int^{N}_{1^-}x^{-s}\operatorname{d}A(x)\\
		&=x^{-s}A(x)\big\vert_{0^-}^{N}-\int_{1^-}^{N}A(x)\operatorname{d}x^{-s}\\
		&=A(N)N^{-s}+s\int_{1^-}^{N}x^{-s-1}A(x)\operatorname{d}x.
	\end{aligned}\]
	Let $\alpha$ be the left-hand side of (\ref{eqn:integral form of D(s) 2}).
	If $\beta>\alpha$, then this says that $A(x)\ll x^\beta$ as $x\to\infty$ where the implied constant may depend on $a_n$ and $\beta$.
	For a complex number $s$, if $\sigma>\beta$, then the integral (\ref{eqn:integral form of D(s) 1}) is absolutely convergent, and
	\[\lim_{N\to\infty}A(N)N^{-s}=0.\]
	
	According to Proposition~\ref{prop:uniform covergence of generating series}, if $\sigma_c<0$, then $\lim_{x\to\infty}A(x)=D_a(0)$ is a finite number, hence $A(x)$ must be bounded.
	So the statement holds when $\sigma_c<0$.
	
	If $\sigma_c\geq0$, then by Proposition~\ref{prop:uniform covergence of generating series}, we know that for any $\beta<\sigma_c$, the series $D_a(s)$ is divergent.
	So, $\alpha$ has to be $\geq\sigma_c$, otherwise the above computation would imply that $D_a(s)$ is convergent at some point such that $\sigma<\sigma_c$.
	On the other hand, choose any $\sigma_0>\sigma_c$. By (\ref{eqn:uniform covergence of generating series 1}) with $s=0$ and $M=0$, we have
	\[A(\imath(N))=-R(\imath(N))(\imath(N))^{\sigma_0}+\sigma_0\int_0^{\imath(N)}R(u)u^{\sigma_0-1}\operatorname{d}u.\]
	Since $R(u)$ is a bounded function, we know that
	\[A(x)\ll x^{\sigma_0}\]
	as $x\to\infty$, where the implied constant may depend on $a_n$ and on $\sigma_0$.
	This shows that $\alpha\leq\sigma_0$, hence $\alpha\leq\sigma_c$.
\end{proof}
Just like Dirichlet series, a generating series $D_a(s)$ also admits the concept of absolute convergence.
To be precise, we say that a generating series $D_a(s)=\sum_{d\in S}a_d d^{-s}$ is \emph{absolutely convergent} if
\[\sum_{d\in S}\lvert a_d\rvert d^{-\sigma}\]
is convergent.
Let $\sigma_a$ be the \emph{abscissa of absolute convergence}.
\begin{lemma}
	If $\delta$ is the abscissa of convergence of the series
	\[\sum_{d}d^{-s},\]
	then $\sigma_c\leq\sigma_a\leq\sigma_c+\delta$.
\end{lemma}
\begin{proof}
	The inequality $\sigma_c\leq\sigma_a$ is clear.
	Let $\epsilon>0$ be a positive real number.
	Since $\sum_{d}a_d d^{-\sigma_c-\epsilon}$ is convergent, we see that $a_d\ll d^{\sigma_c+\epsilon}$ as $d\to\infty$, and the implied constant may depend on $a_d$ and $\epsilon$.
	So the series $\sum_{d}\lvert a_d\rvert d^{-\sigma_c-2\epsilon-\delta}$ must be convergent by comparing with
	\[\sum_{d}d^{-\delta-\epsilon}.\]
\end{proof}
\subsection{Tauberian Theorem}
The work of Delange, generalizing Ikehara's Theorem, could be called ``effective'' or transcendental Tauberian theorem.
Because unlike Hardy-Littlewood Theorem (see \cite[Theorem 5.11]{montgomery2006multiplicative} for example) which relates the behaviour of the real function $\sum_n a_nn^{\sigma}$ when $\sigma>1$ on the real line to that of $\sum_{n\leq x}a_n/n$, it instead cares about the asymptotic information about $\sum_{n<x}a_n$ using the value of the function $\sum_n a_nn^{s}$ when $s$ is complex, which looks like a ``jump'' from the half-plane $\sigma>1$ to the point $s=0$.

From Theorem~\ref{Theorem III} and \ref{Theroem IV} (see also Delange~\cite{Delange54}), we can prove the following.
\begin{theorem}\label{thm:delange}
	Let \(a:S\to\mathbb{R}_{\geq0}\) be a non-negative real function, and let \(A(x):=\sum_{d<x}a_d\) be the summatory function.
	Let \(\alpha\) be the convergence of abscissa of the generating series \(D_a(s)\).
	If there exists holomorphic functions \(g_1,\dots,g_{\gamma}\) and \(h\) in a neighbourhood of the open half-plane \(\sigma>\alpha\), with \(g_{\gamma}(\alpha)\neq0\), such that
	\begin{equation*}
		D_a(s)=(s-\alpha)^{-\beta}\sum_{j=0}^{\gamma}g_j(s)\Bigl(\log\frac{1}{s-\alpha}\Bigr)^j,
	\end{equation*}
	for all \(\Re(s)>\alpha\), where \(\beta,\gamma\) are non-negative integers such that \(\beta+\gamma>0\),
	then as \(t\to\infty\), we have
	\begin{equation*}
		A(x)\sim\left\{\begin{aligned}
			&\frac{g(\alpha)}{\Gamma(\beta)}x^{\alpha}(\log x)^{\beta-1}(\log\log x)^{\gamma}\quad&\text{if }\beta>0;\\
			&\gamma g_{\gamma}(\alpha)\frac{x^{\alpha}}{\log x}(\log\log x)^{\gamma-1}\quad&\text{else if }\beta=0,\gamma\geq1.
		\end{aligned}\right.
	\end{equation*}
\end{theorem}
Note that \(\log s\) is understood to be the main brunch of the logarithm.
\begin{proof}
	Let \(B(t):=\sum_{d<e^t}a_d\).
	For the convenience of our computation, we just assume without loss of generality that \(a_d=0\) for all \(0<d\leq1\), because there are only finitely many \(d\leq1\).
	According to Proposition~\ref{prop:integral form of D(s)}, when \(\Re(s)>\alpha\), we have
	\begin{equation*}
		\begin{aligned}
			D_a(s)=&s\int_0^{\infty}x^{-s-1}A(x)\operatorname{d}x\\
			=&s\int_1^{\infty}x^{-s-1}A(x)\operatorname{d}x\\
			\stackrel{x=e^t}{=}&s\int_0^{\infty}e^{-st}B(t)\operatorname{d}t.
		\end{aligned}		
	\end{equation*}
	Let \(f(s):=\int_0^\infty e^{-st}B(t)\operatorname{d}t\).
	It is clear that \(B(t)\) and \(f(s)\) satisfy the conditions of Theorem~\ref{Theorem III} if \(\alpha>0\), \(\gamma=0\) and \(\beta>0\).
	In this case, we obtain the asymptotic formula
	\begin{equation*}
		B(t)\sim\frac{g(\alpha)}{\Gamma(\beta)}e^{\alpha t}t^{\beta-1}.
	\end{equation*}
	So \(A(x)\sim g(\alpha)\Gamma(\beta)^{-1}x^{\alpha}(\log x)^{\beta-1}\) as in the statement.
	The proof of other cases is just similar and is left to the reader.
\end{proof}
In short, as long as we can prove the analytic continuation together with the pole behaviour at its abscissa of convergence  (which is of course not easy in general), then we can apply the theorem and obtain the asymptotic behaviour of \(A(x)\).
See also Narkiewicz~\cite[Appendix II Theorem I]{narkiewicz2014elementary} for Tauberian Theorem (Delange-Ikehara Theorem) applied to Dirichlet series.
\subsection{Some analytic results}
We prove some results that will be used later.
In the rest of the section, let $\underline{x}\in\mathbb{R}_+^l$ denote an $l$-dimensional vector with positive coordinates, and let
\[S(\underline{x}):=\{d\in\mathbb{R}_+\mid d=e_1^{x_1}e_2^{x_1+x_2}\cdots e_l^{x_1+\cdots+x_l}\text{ where }e_i\text{ is a positive integer}\}\]
be the index set with parameter $\underline{x}$.
Denote by $\underline{u}^n$ a vector of dimension $n$, and omit the superscript $n$ if there is no danger of confusion.
\begin{definition}
	Let
	\begin{equation*}
		U:=\{(\underline{x},s)\mid\underline{x}\in\mathbb{R}_+^l\text{ and }\Re(s)>0\}
	\end{equation*} 
	be an open subset of \(\mathbb{R}^l\times\mathbb{C}\).
	Let $N>0$ be a positive integer number.
	We say that a map $b:\mathcal{P}\to\mathbb{R}_{\geq0}$ is a function modulo $N$ if $b(p)=b(q)$ whenever $p\equiv q\bmod{N}$ for all rational primes $p,q\in\mathcal{P}$.
	A function $a:\mathcal{P}\times U\to\mathbb{C}$ is called a coefficient with main term $b:\mathcal{P}\to\mathbb{R}_+$ if for each \(p\in\mathcal{P}\), \(a(p):U\to\mathbb{C}\) is regular, and there exists $\delta>0$ such that
	\[a(p,\underline{x},s)-b(p)\ll p^{-\delta\sigma}\]
	for $(\underline{x},s)\in U$, where the implied constant is independent of $p$.
\end{definition}
The functions we use involve many variables in general.
But the complex one plays a more important role than others.
So we make the following notation.
\begin{definition}\label{def:multivariable functions}
	If a function \(f(\underline{x},s)\) is continuous in a subset \(U'\) of \(U\), 
	and for each \(\underline{x}\in V\) the function \(f_{\underline{x}}(s):=f(\underline{x},s)\) is a regular function on \(U'_{\underline{x}}:=\{s\in\mathbb{C}\mid(\underline{x},s)\in U'\}\), 
	then we say that \(f:U'\to\mathbb{C}\) is a regular function with parameter \(\underline{x}\).
\end{definition}
\begin{proposition}\label{prop:analytic continuation of zeta(q,a)}
	Let \(a:\mathcal{P}\times U\to\mathbb{C}\) be a coefficient function with main term \(b:\mathcal{P}\to\mathbb{R}_{\geq0}\) being a non-zero function modulo \(N\).
	Define
	    \[D(\underline{x},s):=\prod_{p\in\mathcal{P}}\big(1+a(p,\underline{x},s)p^{-x_1 s}\big).\]
	Then $D(\underline{x},s)$ defines a regular function with parameter \(\underline{x}\) on the subset
	    \[U_0:=\{(\underline{x},s)\in U\mid\Re(s)>x_1^{-1}\},\]
	of $U$, and for each $\underline{x}\in\mathbb{R}_+^l$.
	Let
	\begin{equation*}
		\bar{U}_0:=\{(\underline{x},s)\in U\mid\Re(s)\geq x_1^{-1}\}.
	\end{equation*}
	There exists some regular function $f(\underline{x},s):\bar{U}_0\to\mathbb{C}$ with parameters such that $f(\underline{x},s)\neq0$ for each \((\underline{x},s)\in\bar{U}_0\), and such that
	\[D_{\underline{x}}(s)^{\phi(N)}=f_{\underline{x}}(s)(s-x_1^{-1})^{-\beta}\]
	where $\beta$ is a natural number given by
	    \[\beta=\sum_{\substack{
			n\in(\mathbb{Z}/N)^*\\
			p\equiv n\bmod{N}
	}}b(p).\]
\end{proposition}
\begin{proof}
	According to the definition of the coefficient $a$, let $C>0$ be a real number such that for each $p$ and $(\underline{x},s)\in U$ we have
	    \[\lvert b(p)\rvert<C\text{ and }\lvert a(p,\underline{x},s)\rvert<C\]
	for all $p\in\mathcal{P}$ and $(\underline{x},s)\in U$.
	By the inequality
	    \[\lvert D(\underline{x},s)\rvert\leq\prod_{p\in\mathcal{P}}(1+Cp^{-x_1\sigma}),\]
	we see that $D(\underline{x},s)$ defines a continuous function in $U_0$.
	And for each $\underline{x}\in\mathbb{R}_+^l$, the function $D_{\underline{x}}(s):=D(\underline{x},s)$ is a regular function in the open half-plane $\Re(s)>x_1^{-1}$ by comparing with
	\begin{equation*}
		\prod_{p\in\mathcal{P}}(1+Cp^{-s}).
	\end{equation*}
	Let
	\[\tilde{D}_{\underline{x}}(\underline{s}):=\prod_{p\in\mathcal{P}}\big(1+b(p)p^{-x_1 s}\big).\]
	Similar argument implies that $\tilde{D}(\underline{x},s)$ is also continuous in $U_0$, and for each $\underline{x}\in\mathbb{R}_+^l$, the function $\tilde{D}_{\underline{x}}(\underline{s})$ is a regular function in the open half-plane \(\Re(s)>x_1^{-1}\).
	According to our assumption on the coefficient function $a$, we may assume without loss of generality that $\tilde{D}(\underline{x},\underline{s}),D(\underline{x},\underline{s})$ are both non-zero in $U_0$.
	Otherwise, we only take into account large enough $p$.
	For example, replace $\tilde{D}(\underline{x},\underline{s})$ by 
	\[\prod_{p>C}\bigl(1+b(p)p^{-x_1 s}\bigr).\]  
	Then take logarithmic and we have
	\[\begin{aligned}
		&\lvert\log D(\underline{x},s)-\log\tilde{D}(\underline{x},s)\rvert\\
		=&\Big\lvert\sum_{p\in\mathcal{P}}\sum_{n=1}^\infty\frac{(-1)^{n-1}}{n}a(p,\underline{x},s)^n p^{-nx_1 s}-\sum_{p\in\mathcal{P}}\sum_{n=1}^\infty\frac{(-1)^{n-1}}{n}b(p)^n p^{-nx_1 s}\Big\rvert\\
		=&\Big\lvert\sum_{p\in\mathcal{P}}\sum_{n=1}^\infty\frac{(-1)^{n-1}}{n}\Bigl(a(p,\underline{x},s)^n p^{-nx_1 s}-b(p)^n p^{-nx_1 s}\Bigr)\Big\rvert\\
		\leq&\sum_{p\in\mathcal{P}}p^{-(x_1+\delta)\sigma}+\sum_{p\in\mathcal{P}}\sum_{n=2}^{\infty}2\frac{C^n}{n}p^{-nx_1\sigma}
	\end{aligned}\]
	This implies that the series of the difference $\log D(\underline{x},s)-\log\tilde{D}(\underline{x},s)$ is absolutely convergent on \(\bar{U}_0\), hence a regular function \(\bar{U}_0\to\mathbb{C}\) with parameter \(\underline{x}\).
	Let's denote it by $h(\underline{x},s)$.
	According to \cite[Proposition 5.4]{Wang2022DistributionOT}, there exists some regular functions $g$ such that $g(s)\neq0$ in $\Re(s)\geq1$ and such that
	\[\prod_{p\in\mathcal{P}}(1+b(p)p^{-s})^{\phi(N)}=g(s)(s-1)^{\beta}.\]
	So, we have $\tilde{D}_{\underline{x}}(s)=g(x_1 s)(x_1 s-1)^{\beta}$.
	Finally, for each \(\underline{x}\), we have
	\[D_{\underline{x}}(s)^{\phi(N)}=\bigl(\tilde{D}_{\underline{x}}(s)e^{h(\underline{x},s)}\bigr)^{\phi(N)}=e^{\phi(N)h(\underline{x},s)}g(x_1 s)(x_1 s-1)^{-\beta}.\]
	So we can just take $f(\underline{x},s)$ to be $e^{h(\underline{x},s)}g(x_1 s)x_1^{-\beta}$.
\end{proof}
\begin{proposition}\label{prop:count (q,a) primes}  
	Let \(a_i:\mathcal{P}\times U\to\mathbb{C}\) be a coefficient function with main term \(b_i:\mathcal{P}\to\mathbb{R}_{\geq0}\) being a non-zero function modulo \(N\), where $i=1,\dots,l$.
	Define for \(\gamma=1,2,\dots\) the following series
	\[D_{\underline{t}^{\gamma}}(\underline{x},s):=\sum_{\substack{
			\underline{p}\in\mathcal{P}^{\gamma}\\
			p_1<\cdots<p_{\gamma}
	}}\prod_{i=1}^{\gamma}a_{i}(p,\underline{x},s)p_i^{-t_i(\underline{x})s},\]
    where \(t_i(\underline{x})\) is any projection map \(\pi(\underline{x})=x_j\) for some \(1\leq j\leq l\).
    Let \(x:=\min_{i=1}^{\gamma}\{l_i(\underline{x})\}\).
    \begin{equation*}
    	U_0:=\{(\underline{x},s)\in U\mid x_1<x_2<\cdots<x_l\text{ and }\Re(s)>x^{-1}\}
    \end{equation*}
    be a subset of \(U\).
	Then $D_{\underline{t}^\gamma}(\underline{x},s)$ defines a regular function \(U_0\to\mathbb{C}\) with parameter \(\underline{x}\).
	Let 
	\begin{equation*}
		\bar{U}_0:=\{(\underline{x},s)\in U_0\mid\Re(s)\geq x^{-1}\},
	\end{equation*}
	and let \(\gamma':=\sum_{i:l_i(\underline{x})=x}1\).
	Then there exists some regular functions $f_0(\underline{x},s),\dots,f_{\gamma'}(\underline{x},s):\bar{U}_0\to\mathbb{C}$ with parameters \(\underline{x}\), and $f_{\gamma'}$ can be chosen to be a positive constant, and such that
	\[D_{\underline{t}^\gamma,\underline{x}}(s)=\sum_{i=0}^{\gamma'}f_{i,\underline{x}}(s)\Bigl(\log\frac{1}{s-x^{-1}}\Bigr)^{i}.\]
\end{proposition}
\begin{proof}
	For each $\underline{x}\in\mathbb{R}_+^l$, the series $D(\underline{x},s)$ is absolutely convergent if $\Re(xs)>1$.
	So $D$ defines a continuous function $U_0\to\mathbb{C}$.
	Since each term in the series is holomorphic with respect to $s$, we also see that $D(s)$ is a regular function in $\Re(s)>x^{-1}$ for each $\underline{x}$.
	
	If $\gamma=1$, then 
	\[D_{t}(\underline{x},s)=\sum_{p\in\mathcal{P}}a_1(p,\underline{x},s)p^{-t(\underline{x})s}.\]
	For simplicity, we just assume without loss of generality that \(t(\underline{x})=\pi_1(\underline{x})=x_1\), and write \(D_1:=D_{t}\).
	Define
	\[\tilde{D}_1(\underline{x},s):=\sum_{p\in\mathcal{P}}b_1(p)p^{-x_1s}.\]
	Similar argument also shows that $\tilde{D}_1(\underline{x},s)$ is a continuous function in $U_0$, and for each $\underline{x}\in\mathbb{R}_+^\gamma$, it is a regular function in $\Re(s)>{x_1^{-1}}$.
	Moreover, for each $\underline{x}\in\mathbb{R}_+^\gamma$ and for each $\Re(s)>{x_1^{-1}}$, we have
	\[\lvert D_1(\underline{x},s)-\tilde{D}_1(\underline{x},s)\rvert\ll\sum_{n=1}^\infty p^{-(x_1+\delta_1)\sigma},\]
	which shows that $h(\underline{x},s):=(D_1-\tilde{D}_1)(\underline{x},s)$ is absolutely convergent in $\Re(s)\geq x_1^{-1}$ for each $\underline{x}\in\mathbb{R}_+^{\gamma}$.
	According to \cite[Proposition 5.5]{Wang2022DistributionOT}, we know that there exists some non-zero constant $g_1$ and some regular function $g_0$ in the closed half-plane $\Re(s)\geq1$ such that
	\[\sum_{p\in\mathcal{P}}b_1(p)p^{-s}=g_1(s)\log\frac{1}{s-1}+g_0(s).\]
	So, we have
	\[D_1(\underline{x},s)=h(\underline{x},s)+g_0(x_1s)+g_1(x_1 s)\log\frac{1}{x_1s-1},\]
	which concludes the proof when $\gamma=1$.
	Provided that the statement is true for all $1,\dots,\gamma-1$.
	Let \((\underline{x},s)\in U_0\).
	Define for any coefficient \(a\) and any projection map \(\pi(\underline{x})=x_j\), where \(1\leq j\leq l\), the following notations:
	\[D_{a,\pi}(\underline{x},s):=\sum_{p\in\mathcal{P}}a(p,\underline{x},s)p^{-\pi(\underline{x})s}.\]
	Then
	\begin{equation*}
		\begin{aligned}
			&D_{\underline{l}^{\gamma-1}}(\underline{x},s)D_{a_{\gamma},l_\gamma}(\underline{x},s)\\
			=&D_{\underline{l}^{\gamma}}(\underline{x},s)+\sum_{i=1}^{\gamma-1}D^{(i)}_{\underline{l}^{\gamma-1}}(\underline{x},s)
		\end{aligned}
	\end{equation*}
	where
	\begin{equation*}
		D^{(i)}_{\underline{l}^{\gamma-1}}(\underline{x},s):=\sum_{\substack{
				\underline{p}\in\mathcal{P}^{\gamma-1}\\
				p_1<\cdots<p_{\gamma-1}
		}}a_1p_1^{-l_1(\underline{x})s}\cdots a_{i}a_{\gamma}p_i^{-(l_i+l_{\gamma})(\underline{x})s}\cdots a_{\gamma-1}p^{-l_{\gamma-1}(\underline{x})s}.
	\end{equation*}
	Note that $a_ia_{j}$ is also a coefficient function of main term \(b_ib_{j}\).
	So, for each $i=1,\dots,\gamma-1$, the series \(D^{(i)}_{\gamma-1}(\underline{x},s)\) satisfies the condition of our induction assumption.
	So the analytic properties of $D$ is proved by induction on $\gamma$.
\end{proof}

\section{Abelian fields}
Let $G$ be a finite abelian group in this section.
By viewing $G$ as a transitive permutation group according to its action on itself, we see that $\mathcal{S}:=\mathcal{S}(G)$ is just the set of abelian $G$-fields.
In Section~\ref{section:invariant} Definition~\ref{def:invariant general} and Definition~\ref{def:invariant of abelian fields}, we've explained how to define a counting function \(\Theta\) of $\mathcal{S}$, and we are going to prove some results on counting fields with respect to \(\mathcal{S}(\Theta)\), the set of abelian \(G\)-fields ordered by \(\Theta\).
\subsection{Invariant on abelian fields}
Since we require that the weight function \(\nu:G\to\mathbb{R}_{\geq0}\) satisfy the condition that \(\nu(g)=\nu(h)\) if \(g\sim h\) (equivalent under invertible powering), we see that \(\nu^{-1}(x)\) where \(x\in\mathbb{R}_{\geq0}\) gives a partition of the group \(G\).
Of course, it may happen that \(\nu(g)=\nu(h)\) even if \(g\) and \(h\) are not equivalent.
For example \(\nu(g)=1\) for all \(g\neq1\).
In this section, we restrict ourselves to the case when the partition
\begin{equation*}
	G=\bigcup_{x\geq0}\nu^{-1}(x)
\end{equation*}
does not change.
To be precise, we make the following notation.
\begin{definition}
	Let \(G/\sim:=\{1,g_1,\dots,g_{l'}\}\) be a set of representatives of the equivalence classes of \(G\) under invertible powering.
	Let \(\{S_0=\{1\},S_1,\dots,S_l\}\) be a set of subsets of \(G/\sim\) such that
	\begin{equation*}
		G/\sim=\bigsqcup_{i=0}^l S_i.
	\end{equation*}
	We call \(\nu\) a weight function with respect to \(\{S_0,\dots,S_l\}\) if \(\nu(g_i)=\nu(h_i)\) for all \(g_i,h_i\in S_i\), where \(i=1,\dots,l\), and if \(\nu(g_i)<\nu(g_j)\) for all \(g_i\in S_i\) and \(g_j\in S_j\), where \(0\leq i<j\leq l\).
	In this case, write \(x_i:=\nu(g_i)\) for all \(g_i\in S_i\) and \(i=1,\dots,l\).
\end{definition}
A simple example is, of course, \(\{S_0=\{1\},S_1=\{g_1,\dots,g_{l'}\}\}\), and \(\nu(g)=1\) for all \(g\neq1\).
\begin{notation}
	Fix a partition \(G/\sim=\bigsqcup_{i=0}^l S_i\) and let \(\nu\) be a weight function with respect to \(\{S_0,\dots,S_l\}\).
	Define the following subsets:
	\begin{equation*}
		\begin{aligned}
			V:=&\{(x_1,\dots,x_l)\in\mathbb{R}_+^l\mid x_1<x_2<\cdots<x_l\}\\
			U:=&\{(x_1,\dots,x_l,s)\in V\times\mathbb{C}\mid\sigma>0\}\\
			U_0:=&\{(x_1,\dots,x_l,s)\in U\mid\sigma>x_1^{-1}\}\\
			\bar{U}_0:=&\{(x_1,\dots,x_l,s)\in U_0\mid\sigma\geq x_1^{-1}\}.
		\end{aligned}
	\end{equation*}
\end{notation}
To determine \(\Theta\), we only need to specify \(\vartheta\) at primes \(p\mid\lvert G\rvert\).
For simplicity, we just define \(\vartheta\) for the morphisms \(\rho:\mathbb{Z}_p^*\to G\).
\begin{definition}\label{def:invariant of morphisms}
	Let \(p\nmid\infty\) be a finite prime of \(\mathbb{Q}\).
	If \(\rho:\mathbb{Z}_p^*\to G\), is a local homomorphism, then define
	\begin{equation*}
		\vartheta(\rho):=\min_{\substack{
				T\subseteq\mathbb{Z}_p^*\\
				\langle T\rangle=\mathbb{Z}_p^*
			}}\sum_{g\in T}\nu(g).
	\end{equation*}
	For any lift \(\tilde{\rho}:\mathbb{Q}_p^*\to G\), just let \(\vartheta(\tilde{\rho}):=\vartheta(\rho)\).
\end{definition}
Of course, this notation is consistent when we apply it to \(\Sigma_p\) when \(p\nmid\lvert G\rvert\infty\), because \(\Sigma_p\) corresponds to a local map \(\rho_p:\mathbb{Q}_p^*\to G\), and the tame inertia of \(\Sigma_p\) is exactly determined by the image of \(\mu(\mathbb{Q}_p)\).
Note that by Lemma~\ref{lemma:CFT of Q}, we know that \(\operatorname{Hom}(\prod_{p\nmid\infty}\mathbb{Z}_p^*,G)\) corresponds to \(\operatorname{Hom}(\operatorname{C}_{\mathbb{Q}},G)\), and surjective ones corresponds to Artin reciprocity maps.
Now for each Artin reciprocity map \(\rho\), we have
\begin{equation*}
	\Theta(\rho)=\vartheta(\rho\vert_p).
\end{equation*}
\subsection{Statement of the results on counting fields}
Let \(\Omega\) be a non-empty subset of \(G\) that is closed under invertible powering.
Since each \(\vartheta\) is computed by \(\nu\), we know that \(\Theta=\Theta_{\underline{x}}\) is a counting function with parameter \((x_1,\dots,x_l)\in\mathbb{R}_+\).
So we just denote by \(\mathcal{S}(\underline{x})\) the set of abelian \(G\)-fields ordered by \(\Theta\).
Write
\begin{equation*}
	N_{\underline{x}}(t):=N(\mathcal{S}(\underline{x});t).
\end{equation*}
For each integer \(\gamma\geq0\) let
\begin{equation*}
	N_{\underline{x},\gamma}(t):=N(\mathcal{S}(\underline{x});(\Omega,\gamma);t).
\end{equation*}
It will be shown later that the value \(\min\{\nu(g)\mid\langle g\rangle\cap\Omega\neq\emptyset\}\) compared to \(x_1\) is very important.
So let's make the following notation.
\begin{definition}
	Define
	\begin{equation*}
		x:=\min\{\nu(g)\mid\langle g\rangle\cap\Omega=\emptyset\}\quad\text{and}\quad y:=\min\{\nu(g)\mid\langle g\rangle\cap\Omega\neq\emptyset\}.
	\end{equation*}
\end{definition}
Clearly \(x_1=\min\{x,y\}\).
It is clear that \(\Theta(K)\) is contained in the index set
\[S(\underline{x}):=\{d=e_1^{x_1}\cdots e_l^{x_l}\text{ where }e_1,\dots,e_l\in\mathbb{Z}_+\}\]
for all \(K\in\mathcal{S}\).
We need a technical notation to describe the difference between $\operatorname{Hom}(\prod_p\mathbb{Z}_p^*,G)$ and $\operatorname{Sur}(\prod_p\mathbb{Z}_p^*,G)$ when it comes to field counting.
For example, in the sense of asymptotic behaviour, there is no difference between these two concept when $G\cong C_q$ where $q$ is a rational prime.
But in general, the result of counting non-surjective homomorphisms may give a even larger main term than counting surjective ones.
\begin{definition}\label{def:delta0}
	\begin{enumerate}
		\item Let \(\rho:\prod_{p\nmid\infty}\mathbb{Z}_p^*\to G\) be a surjective homomorphism.
		Define
		\begin{equation*}
			\delta(\rho):=\#\{p\nmid\lvert G\rvert:\mathbf{1}_{(\Omega)}(\rho\vert_p)=1\text{ and }\vartheta_{\underline{x}}(\rho\vert_p)>p^{y}\}.
		\end{equation*}
		Let \(\delta:=\min_{\rho}\{\delta(\rho)\}\) where \(\rho:\prod_{p\nmid\infty}\mathbb{Z}_p^*\to G\) runs over all surjective homomorphisms.
		\item Let \(\gamma(\rho)\) be the integer such that \(\mathbf{1}_{(\Omega,\gamma)}(\rho)=1\), and let \(\gamma_0:=\min_{\rho}\gamma(\rho)\) where \(\rho\) runs over all surjective homomorphisms such that \(\delta(\rho)=\delta\).
	\end{enumerate}	
\end{definition}
Inspired by Wood~\cite[Theorem 3.1]{wood2010probabilities}, we make the following definition.
\begin{definition}\label{def:count abelian fields}
	For an element $g\in G$, we denote its order by $r(g)$.
	Let $\Lambda$ be a subset of $G$ closed under invertible powering. Define
	\[\beta(\Lambda):=\sum_{\operatorname{id}\neq g\in\Lambda}\phi(r(g))^{-1}.\]
	In particular, let \(\mathfrak{M}:=\{g\in G\mid\nu^{-1}(x)\text{ and }\langle g\rangle\cap\Omega=\emptyset\}\), and define
	\[\beta:=\beta(\mathfrak{M}).\]		
\end{definition}
Since our discussion is based on the weight functions associated to a specific choice \(G/\sim=\bigsqcup_{i=0}^lS_i\), we see that \(\mathfrak{M}\) remains the same for all \(\underline{x}\in V\) and \(\beta\) is just a constant (independent of the choice of \(\underline{x}\in V\)).
The main theorem of this section can be stated as follows.
\begin{theorem}\label{thm:counting abelian fields}
	We have that
	\begin{equation*}
		N_{\underline{x}}(t)\sim F(\underline{x})t^{x_1^{-1}}(\log t)^{\beta-1},
	\end{equation*}
	when \(t\to\infty\), where \(F\) is a non-zero continuous function on \(V\).
	Let \(\gamma>\gamma_0\) be an integer.
	We have that
	\begin{equation*}
		N_{\underline{x},\gamma}(t)\sim\left\{\begin{aligned}
			&f(\underline{x})t^{x_1^{-1}}(\log t)^{-1}(\log\log t)^{\gamma-\delta-1}\quad&\text{if }y<x\\
			&g(\underline{x})t^{x_1^{-1}}(\log t)^{\beta-1}(\log\log t)^{\gamma-\delta'}\quad&\text{else if }y=x\\
			&h(\underline{x})t^{x_1^{-1}}(\log t)^{\beta-1}\quad&\text{else if }y>x,
		\end{aligned}\right.
	\end{equation*}
	where \(\delta'=\delta\) if \(\beta>0\) and \(\delta'=\delta+1\) otherwise, and \(f,g,h\) are all continuous functions on \(V\).
\end{theorem}
Recall also from Definition~\ref{def:function R} that we have defined a function $R$ (or the expression): 
\[R^{\gamma_1}_{\gamma_2}(\underline{x})=\lim_{t\to\infty}\frac{N_{\underline{x},\gamma_1}(t)}{N_{\underline{x},\gamma_2}(t)}.\]
\begin{theorem}\label{thm:function R abelian case}
	Let $\Omega\subseteq G$ be a nonempty subset closed under invertible powering.
	Let $\gamma_2>\gamma_1\geq\gamma_{0}$ be two integers.
	Then $R^{\gamma_1}_{\gamma_2}(\underline{x})$ is a continuous function on $V$.
	In addition, if $y>x_1$, then $R^{\gamma_1}_{\gamma_2}(\underline{x})>0$.
	Else if $y=x_1$, then $R^{\gamma_1}_{\gamma_2}(\underline{x})=0$.
\end{theorem}
\begin{proof}
	If $y>x$, according to Theorem~\ref{thm:counting abelian fields}, we have
	\[\begin{aligned}
		R^{\gamma_1}_{\gamma_2}(\underline{x})=&\lim_{t\to\infty}\frac{N_{\underline{x},\gamma_1}(t)}{N_{\underline{x},\gamma_2}(t)}\\
		=&\lim_{t\to\infty}\frac{f_{\gamma_1}(\underline{x})\Gamma(\beta)^{-1}t^{x_1^{-1}}(\log t)^{\beta-1}}{f_{\gamma_2}(\underline{x})\Gamma(\beta)^{-1}t^{x_1^{-1}}(\log t)^{\beta-1}}\\
		=&\frac{f_{\gamma_1}(\underline{x})}{f_{\gamma_2}(\underline{x})}
	\end{aligned}\]
	We can see that $R^{\gamma_1}_{\gamma_2}(\underline{x})>0$ because $f_{\gamma_i}(\underline{x})>0$ for each $\underline{x}\in V$.
	Moreover, $f_{\gamma_i}(\underline{x}):V\to\mathbb{C}$ is a continuous function, so is $R^{\gamma_1}_{\gamma_2}(\underline{x})$.
	Similar argument works for the case when $y=x$ and \(y<x\) except that $R^{\gamma_1}_{\gamma_2}(\underline{x})=0$ for all $\underline{x}\in V$ in these cases.
\end{proof}
\subsection{Generating series}
Recall that for integers $m,n$ such that \(\gcd(m,n)=1\), we've defined the notations 
\begin{equation*}
	\mathcal{P}(m,n)=\{p\in\mathcal{P}\mid p\equiv m\bmod{n}\}\quad\text{and}\quad\mathcal{P}(n)=\{p\in\mathcal{P}\mid p\nmid n\}.
\end{equation*}
We define for each $\underline{x}\in V$ and each natural number $\gamma=0,1,\dots$, the following generating series
\begin{equation*}
	\mu_{\gamma}(\underline{x},s):=
	\sum_{\rho:\prod_{p\nmid\infty}\mathbb{Z}_p^*\to G}
	\mathbf{1}_{(\Omega,\gamma)}(\rho)(\Theta_{\underline{x}}(\rho))^{-s}
	\quad\text{and}\quad
	\mu(\underline{x},s):=
	\sum_{\rho:\prod_{p\nmid\infty}\mathbb{Z}_p^*\to G}
	(\Theta_{\underline{x}}(\rho))^{-s}
\end{equation*}
And let
\[\pi_{\gamma}(\underline{x},s):=
\sum_{K\in\mathcal{S}}
\mathbf{1}_{(\Omega,\gamma)}(K)(\Theta_{\underline{x}}(K))^{-s}
\quad\text{and}\quad
\pi(\underline{x},s):=
\sum_{K\in\mathcal{S}}
(\Theta_{\underline{x}}(K))^{-s}.\]
For the purpose of simplifying the notations, let \(\mu_{-1}:=\mu\) and \(\pi_{-1}:=\pi\).
Apparently \(\mathbf{1}_{(\Omega,-1)}(K)=0\) for all \(K\in\mathcal{S}\), so it represents the trivial condition, and we can re-define the indicator in this case as
\begin{equation*}
	\mathbf{1}_{(\Omega,-1)}(K)=1
\end{equation*}
for all \(K\in\mathcal{S}\) so that everything is consistent.
We can do the same for all the homomorphisms \(\rho\in\operatorname{Hom}(\prod\mathbb{Z}_p^*,G)\), i.e.,
\begin{equation*}
	\mathbf{1}_{(\Omega,-1)}(\rho)=1.
\end{equation*}
From Definition~\ref{def:invariant of morphisms}, Lemma~\ref{lemma:CFT of Q}, we can obtain the arithmetic properties of \(\mu_{\gamma}\) and \(\pi_{\gamma}\).
\begin{lemma}\label{lemma:arithmetic of mu and pi}
	Let \(\tilde{\mathcal{S}}(\underline{x}):=(\operatorname{Hom}(\prod_{p\nmid\infty}\mathbb{Z}_p^*,G),\Theta)\) be the set of homomorphisms ordered by \(\Theta\).
	For each $\underline{x}\in V$, and for \(\gamma=-1,0,1,\dots\) if we write 
	\[\mu_{\gamma}(\underline{x},s)=\sum_{d\in S(\underline{x})}a_{d}d^{-s}\quad\text{and}\quad\pi_{\gamma}(\underline{x},s)=\sum_{d\in S(\underline{x})}b_{d}d^{-s},\]
	then
	\[\begin{aligned}
		&\tilde{N}_{\underline{x},\gamma}(t):=N(\operatorname{Hom}(\tilde{\mathcal{S}}(\underline{x});(\Omega,\gamma);t)=\sum_{d<t}a_{d}\\
		&N_{\underline{x},\gamma}(t)=\sum_{d<t}b_{d}.
	\end{aligned}\]
	In addition, $\mu_{\gamma}(\underline{x},s)\geq\pi_{\gamma}(\underline{x},s)$, i.e., $a_d\geq b_d$ for all $d\in S$.
\end{lemma}
Then we can state the result on the analytic continuation of $\pi_{\gamma}(s)$.
\begin{theorem}\label{thm:analytic continuation of pi(s)}
	For each natural number $\gamma=-1,0,1,\dots$, the expression $\pi_{\gamma}(\underline{x},s):U_0\to\mathbb{C}$ defines a regular function with parameter \(\underline{x}\) (see Definition~\ref{def:multivariable functions}).
	If \(\gamma=-1\) or \(\gamma>\gamma_0\) and $y>x$, then there exists regular functions $f(\underline{x},s),f_0(\underline{x},s):\bar{U}_0\to\mathbb{C}$ with parameter \(\underline{x}\) such that \(f_0(\underline{x},x_1^{-1})\neq0\) and such that
	\[\pi_{\gamma,\underline{x}}(s)=f_{\underline{x}}(s)+f_{0,\underline{x}}(s)(s-x_1^{-1})^{\beta(\nu^{-1}(x_1))}.\]
	Else if \(\gamma>\gamma_0\) and $y\leq x$, then there exists regular functions $f(\underline{x},s),f_0(\underline{x},s),\dots,f_{\gamma-\delta}(\underline{x},s):\bar{U}_0\to\mathbb{C}$ with parameter \(\underline{x}\) such that \(f_{\gamma-\delta}\) could be chosen to be a positive constant and such that
	\[\pi_{\gamma,\underline{x}}(s)=f_{\underline{x}}(s)+(s-x^{-1})^{\beta}\sum_{j=0}^{\gamma-\delta}f_{j,\underline{x}}(s)\big(\log\frac{1}{s-y^{-1}}\big)^{j}.\]
\end{theorem}

\subsection{Analytic properties of \texorpdfstring{$\mu$}{mu}}
We prove the statement of Theorem~\ref{thm:analytic continuation of pi(s)} step by step.
The series \(\mu_{\gamma}\), in the realm of absolute convergence, could be written as an Euler product in the following sense.
\begin{lemma}\label{lemma:euler product of mu}
	In the set \(U_0\), when \(\gamma\geq0\), we have
	\begin{equation*}
		\mu_{\gamma}(\underline{x},s)=W(\underline{x},s)\sum_{\substack{
				\underline{p}\in\mathcal{P}(\lvert G\rvert)^{\gamma}\\
				p_1<\cdots<p_{\gamma}
		}}\prod_{j=1}^{\gamma}\Bigl(\sum_{\rho:\mathbb{Z}_p^*\to G}\mathbf{1}_{\Omega}(\rho)p_j^{-\vartheta_{\underline{x}}(\rho)s}\Bigr)\cdot\prod_{\substack{
				p\in\mathcal{P}(\lvert G\rvert)\\
				p\nmid p_1\cdots p_\gamma
		}}\Bigl(\sum_{\rho:\mathbb{Z}_p^*\to G}(1-\mathbf{1}_{\Omega}(\rho))p^{-\vartheta_{\underline{x}}(\rho)s}\Bigr),
	\end{equation*}
	where \(W(\underline{x},s):U\to\mathbb{C}\) is a regular function with parameter \(\underline{x}\).
	Similarly results hold for \(\mu=\mu_{-1}\), i.e., there exists some \(W\) such that
	\begin{equation*}
		\mu(\underline{x},s)=W(\underline{x},s)
		\sum_{p\nmid\lvert G\rvert\infty}\prod_{j=1}^{\gamma}\Bigl(\sum_{\rho:\mathbb{Z}_p^*\to G}p_j^{-\vartheta_{\underline{x}}(\rho)s}\Bigr).
	\end{equation*}
\end{lemma}
\begin{proof}
	By abuse of notation, we drop the restriction that \(\Omega\neq\emptyset\) in the proof and let \(\mu(\underline{x})=\mu_{\gamma}(\underline{x},s)\) for all \(\gamma=0,1,\dots\) when \(\Omega=\emptyset\).
	The function \(\mu_{\gamma}\) can be first rewritten as
	\begin{equation*}
		\begin{aligned}
			\mu_{\gamma}(\underline{x},s)=&
			\Bigl(\sum_{\rho':\prod_{p\mid\lvert G\rvert}\mathbb{Z}_p^*\to G}\mathbf{1}_{(\Omega,\gamma)}^{\delta(\gamma)}(\rho')(\Theta_{\underline{x}}(\rho'))^{-s}\Bigr)
			\Bigl(\sum_{\rho'':\prod_{p\nmid\lvert G\rvert\infty}\mathbb{Z}_p^*\to G}\mathbf{1}_{(\Omega,\gamma)}(\rho'')(\Theta_{\underline{x}}(\rho''))^{-s}\Bigr)\\
			=:&W(s)\cdot\sum_d a_dd^{-s},
		\end{aligned}		
	\end{equation*}
	where \(\delta(x)=1\) when \(x=0\) and \(\delta(x)=0\) otherwise.
	Because every homomorphism \(\rho:\prod_{p\nmid\infty}\mathbb{Z}_p^*\to G\) is a product of \(\rho':\prod_{p\mid\lvert G\rvert}\mathbb{Z}_p^*\to G\) and \(\rho'':\prod_{p\nmid\lvert G\rvert\infty}\mathbb{Z}_p^*\to G\).
	Since the number of primes \(p\mid\lvert G\rvert\) is finite, we see that \(W\) is a finite sum of the form
	\begin{equation*}
		\Theta_{\underline{x}}(\rho')^{-s}=\prod_{p\mid\lvert G\rvert}p^{-\vartheta_{\underline{x}}(\rho'\vert_p)s}.
	\end{equation*}
	So, \(W\) satisfies the properties described in the statement of the lemma.
	It then suffices to show that the identity
	\begin{equation*}
		\sum_{d\in S(\underline{x})}a_dd^{-s}=\sum_{\substack{
				\underline{p}\in\mathcal{P}(\lvert G\rvert)^{\gamma}\\
				p_1<\cdots<p_{\gamma}
		}}\prod_{j=1}^{\gamma}\Bigl(\sum_{\rho:\mathbb{Z}_p^*\to G}\mathbf{1}_{\Omega}(\rho)p_j^{-\vartheta_{\underline{x}}(\rho)s}\Bigr)\cdot\prod_{\substack{
				p\in\mathcal{P}(\lvert G\rvert)\\
				p\nmid p_1\cdots p_\gamma
		}}\Bigl(\sum_{\rho:\mathbb{Z}_p^*\to G}(1-\mathbf{1}_{\Omega}(\rho))p^{-\vartheta_{\underline{x}}(\rho)s}\Bigr).
	\end{equation*}	
	If \(\Omega=\emptyset\), then
	\begin{equation*}
		\sum_{\Theta_{\underline{x}}(\rho)=d}1=\sum_{\substack{
				\rho_{p_1},\cdots,\rho_{p_n}\text{ nontrivial}\\
				\prod\vartheta_{\underline{x}}(\rho_{p_j})=d
		}}1=a_d.
	\end{equation*}
	where the sum is taken over all possible nontrivial local homomorphisms \(\rho_{p_j}:\mathbb{Z}_{p_j}^*\to G\) up to permutation.
	On the other hand, by similar idea, we have
	\begin{equation*}
		\sum_{\substack{
				\Theta_{\underline{x}}(\rho)=d\\
				\mathbf{1}_{(\Omega,\gamma)}(\rho)=1
		}}1=\sum_{\substack{
				\underline{p}\in\mathcal{P}(\lvert G\rvert)^{\gamma}\\
				p_1<\cdots<p_{\gamma}
		}}\sum_{\substack{
				\underline{q}\in\mathcal{P}(p_1\cdots p_{\gamma})^n\\
				q_1<\cdots<q_n
		}}\sum_{\substack{
				\prod\vartheta_{\underline{x}}(\rho_{p_i})\\
				\cdot\prod\vartheta_{\underline{x}}(\rho_{q_j})=d
		}}\prod_{i=1}^{\gamma}\mathbf{1}_{\Omega}(\rho_{p_i})\prod_{j=1}^n(1-\mathbf{1}_{\Omega}(\rho_{q_j}))=a_d		
	\end{equation*}
	where \(n\geq0\) is a non-negative integer, and \(\rho_{p_i},\rho_{q_j}\) runs over all non-trivial local homomorphisms.
	And we are done.
\end{proof}
Let's prove the analytic continuation of \(\mu\) and $\mu_{\gamma}$.
\begin{lemma}\label{lemma:analytic continuation of mu 2}
	The expression $\mu_{\gamma}(\underline{x},s):U_0\to\mathbb{C}$ defines a regular function with parameter \(\underline{x}\).
	There exists regular functions $f(\underline{x},s),g_0(\underline{x},s),\dots,g_{\gamma}(\underline{x},s):\bar{U}_0\to\mathbb{C}$ with parameter \(\underline{x}\) such that $g_{\gamma}(\underline{x},y^{-1})$ is some positive constant and $f(\underline{x},x^{-1})\neq0$, and such that
	\[\mu_{\gamma,\underline{x}}(s)=f_{\underline{x}}(s)(s-x^{-1})^{\beta}\sum_{j=0}^{\gamma}g_{j,\underline{x}}(s)\big(\log\frac{1}{s-y^{-1}}\big)^j.\]
\end{lemma}
\begin{proof}
	Let's first rewrite the series.
	By comparing with a large enough power of Riemann-zeta function, we see that for each $\underline{x}$, the complex function $\mu_{\gamma}(\underline{x},x_1s)$ is absolutely convergent in $\sigma>1$, hence $\mu_{\gamma}(\underline{x},s)$ is a regular function in the open half-plane $\sigma>x_1^{-1}$.
	We first write
	\begin{equation*}
		\sum_{\rho:\mathbb{Z}_p^*\to G}(1-\mathbf{1}_{\Omega}(\rho))p^{-\vartheta_{\underline{x}}(\rho)s}=a(p,\underline{x},s)p^{-x s},
	\end{equation*}
	where $a$ is a coefficient function of main term
	\[\sum_{\substack{
			\rho:\mathbb{Z}_p^*\to G\\
			\vartheta_{\underline{x}}(\rho)=x
	}}1-\mathbf{1}_{\Omega}(\rho).\]
    It is clear that this main term is a function modulo \(\lvert G\rvert\).
	Then in $U$, we can rewrite the series as
	\[\begin{aligned}
		\mu_{\gamma}(\underline{x},s)=&W(\underline{x},s)\prod_{p\in\mathcal{P}(\lvert G\rvert)}(1+a(p,\underline{x},s)p^{-x_0 s})\\
		\cdot&\prod_{\substack{
				\underline{p}\in\mathcal{P}^{\gamma}(\lvert G\rvert)\\
				p_1<\cdots<p_{\gamma}
		}}\prod_{j=1}^{\gamma}\Bigl(\sum_{\rho:\mathbb{Z}^*_{p_j}\to G}\frac{\mathbf{1}_{\Omega}(\rho)p_j^{-\vartheta_{\underline{x}}(\rho)s}}{1+a(p_j,\underline{x},s)p_j^{-x_0 s}}\Bigr).
	\end{aligned}\]
	Since for each $p\nmid\lvert G\rvert$, we have
	\begin{equation*}
		\begin{aligned}
			\sum_{\rho:\mathbb{Z}_p^*\to G}\frac{\mathbf{1}_{\Omega}(\rho)p^{-\vartheta_{\underline{x}}(\rho)s}}{1+a(p,\underline{x},s)p^{-x_0 s}}=&p^{-ys}\frac{\sum_{\rho:\mathbb{Z}_p^*\to G}p^{-(\vartheta_{\underline{x}}-y)s}}{1+a(p,\underline{x},s)p^{-x_0 s}}\\
		=&b(p,\underline{x},s)p^{-ys}
		\end{aligned}
	\end{equation*}
	where \(b(p,\underline{x},s)\) is a coefficient function of main term
	\[\sum_{\substack{
			\rho:\mathbb{Z}_p^*\to G\\
			\vartheta_{\underline{x}}(\rho)=y
		}}1.\]
	Similarly, this main term is also a function modulo \(\lvert G\rvert\).
	By the map \(x_i\mapsto x_i'\) where \(x_1=x_1\) and \(x_i'=x_i-x_1\) for all \(i=2,\dots,l\), and by Proposition~\ref{prop:count (q,a) primes}, we see that there exists $g_0,\dots,g_{\gamma}$ satisfying the conditions in the statement of the lemma, such that
	\[\prod_{\substack{
			\underline{p}\in\mathcal{P}^{\gamma}(\lvert G\rvert)\\
			p_1<\cdots<p_{\gamma}
	}}\prod_{j=1}^{\gamma}b(p_j,\underline{x},s)p_j^{-ys}=\sum_{j=1}^{\gamma}g_j(\underline{x},s)\Bigl(\log\frac{1}{s-y^{-1}}\Bigr)^{\gamma}.\]
    Then Proposition~\ref{prop:analytic continuation of zeta(q,a)} finishes the proof except for the formula of analytic continuation.
    Proposition~\ref{prop:analytic continuation of zeta(q,a)} implies that
    \[\prod_{p\in\mathcal{P}(\lvert G\rvert)}(1+a(p,\underline{x},s)p^{-x s})^{\phi(\lvert G\rvert)}=g(\underline{x},s)(s-x^{-1})^{\beta'},\]
    where $g:\bar{U}_0\to\mathbb{C}$ is a regular function with parameter and non-zero everywhere, and
    \[\begin{aligned}
    	\beta'=&\sum_{\substack{
    			n\in(\mathbb{Z}/\lvert G\rvert)^*\\
    			p\equiv n\bmod{\lvert G\rvert}
    	}}\sum_{\substack{
    	    \rho:\mathbb{Z}_p^*\to G\\
    	    \vartheta_{\underline{x}}(\rho)=x
    	}}1-\mathbf{1}_{\Omega}(\rho)=\sum_{n\in(\mathbb{Z}/\lvert G\rvert)^*}\sum_{\substack{
        g\in\mathfrak{M}\\
        n\equiv1\bmod{\phi(r(g))}
    }}1\\
   	=&\sum_{g\in\mathfrak{M}}\sum_{\substack{
   			n\in(\mathbb{Z}/\lvert G\rvert)^*\\
   			n\equiv1\bmod{\phi(r(g))}
   		}}1=\sum_{g\in\mathfrak{M}}\frac{\phi(\lvert G\rvert)}{\phi(r(g))}=\phi(\lvert G\rvert)\beta.
    \end{aligned}\]
	So there exists $f$ satisfying the desired properties in the statement such that
	\[W(\underline{x},s)\prod_{p\in\mathcal{P}(\lvert G\rvert)}(1+a(p,\underline{x},s)p^{-x s})=f(\underline{x},s)(s-x^{-1})^{-\beta}.\]
	And we are done for the proof.
\end{proof}
\subsection{Estimate of field-counting}
Then we construct some generating series to estimate the bounds of field-counting.
We need to construct a generating series $\psi_{\gamma}(s)$ such that $\psi_{\gamma}(s)\geq\pi_{\gamma}(s)$.
Because $\mu_{\gamma}(s)$ gives too large an estimate for $N_{\underline{x}}(t)$ when $y=x_1$.
Let's introduce some notations.
\begin{definition}\label{def:admissible partition}
	Let \(\bar{\Omega}:=\{g\in G\mid\langle g\rangle\cap\Omega\neq\emptyset\}\), and let \(\{h_1,\dots,h_\omega\}\) be the set of representatives of \(\bar{\Omega}\) under the equivalence relation given by invertible powering.
	For each homomorphism \(\rho:\prod_{p\nmid\infty}\mathbb{Z}_p^*\to G\), and for each \(i=1,\dots,\omega\), define 
	\begin{equation*}
		n_i(\rho):=\#\{p\nmid\lvert G\rvert\infty:\rho(\zeta_p)\sim h_i\}.
	\end{equation*}
	If \(\gamma\geq\gamma_0\) is an integer, then we say that the partition \(\gamma=n_1+\cdots+n_{\omega}\) into non-negative integers is \emph{admissible} if there exists some surjective \(\rho\in\operatorname{Sur}(\prod_{p\nmid\lvert G\rvert}\mathbb{Z}_p^*,G)\) such that \(n_i(\rho)=n_i\) for all \(i=1,\dots,\omega\).
\end{definition}
\begin{lemma}\label{lemma:admissible partition}
	If \(\gamma\geq\gamma_0\), then there exists surjective homomorphisms \(\rho:\prod_{p\nmid\infty}\mathbb{Z}_p^*\to G\) such that \(\delta(\rho)=\delta\).
\end{lemma}
\begin{proof}
	It is clear that if \(\gamma=\gamma_0\), then there exists at least one \(\rho_1\in\operatorname{Sur}(\prod_{p\nmid\infty}\mathbb{Z}_p^*,G)\) such that \(\gamma(\rho_1)=\gamma\), which is just the definition of \(\gamma_0\).
	If \(\gamma>\gamma_0\), then let \(\gamma':=\gamma-\gamma_0\).
	Choose finite primes \(p_1,\dots,p_{\gamma'}\) such that \(p_i\equiv1\bmod\lvert G\rvert\), for all \(i=1,\dots,\gamma'\).
	Let \(g\in\Omega\) be any element such that \(\nu(g)=y\).
	For each \(p_i\),  let \(\rho_{p_i}:\mathbb{Z}_{p_i}^*\to G\) be the homomorphism defined by \(\rho_{p_i}(\zeta_{p_i})=g\).
	Define \(\rho:=\rho_1\cdot\prod_{i=1}^{\gamma'}\rho_{p_i}\).
	Then \(\gamma(\rho)=\gamma_0+\gamma'=\gamma\).
	Moreover, \(\vartheta(\rho\vert_{p_i})=p^{y}\) for all \(i=1,\dots,\gamma'\).
	Therefore \(\delta(\rho)=\delta(\rho_1)=\delta\).
\end{proof}
\begin{lemma}\label{lemma:psi}
	Let \(\gamma>\gamma_0\) be a fixed integer.
	There exists a generating series \(\psi_{\gamma}(\underline{x},s)=\sum_{d}a_dd^{-s}\) such that \(\psi_{\gamma}(\underline{x},s)\geq\pi_{\gamma}(\underline{x},s)\) and such that it defines a regular function \(U_0\to\mathbb{C}\) with parameter \(\underline{x}\).
	Moreover, there exists regular functions $f,g_{0},\dots,g_{\gamma-\delta}:\bar{U}_0\to\mathbb{C}$ with parameter such that $g_{\gamma-\delta}$ can be chosen to be a non-zero constant and $f_{\underline{x}}(x^{-1})\neq0$, and such that
	\[\psi_{\gamma,\underline{x}}(s)=f_{\underline{x}}(s)(s-x^{-1})^{-\beta}\sum_{j=0}^{\gamma-\delta}g_{j,\underline{x}}(s)\big(\log\frac{1}{s-y^{-1}}\big)^{j}.\]
\end{lemma}
\begin{proof}
	Let's construct a generating series as follows.
	\begin{equation*}
		\psi_{\gamma}(\underline{x},s):=\sum_{\substack{
				\rho:\prod_{p\nmid\infty}\mathbb{Z}_p^*\to G\\
				\gamma=\sum n_i(\rho)\text{ admissible}
			}}\Theta_{\underline{x}}(\rho)^{-s}.
	\end{equation*}
	From the definition we see that \(\psi_{\gamma}(\underline{x},s)\geq\pi_{\gamma}(\underline{x},s)\) for each \(\underline{x}\in V\), because every surjective homomorphism \(\rho\) satisfies the condition that \(\gamma=\sum n_i(\rho)\) is admissible.
	Let \(\gamma=n_1+\cdots+n_\omega\), with \(\underline{n}:=(n_1,\dots,n_\omega)\), be an admissible partition, and let \(\underline{n}(\rho):=(n_1(\rho),\dots,n_\omega(\rho))\).
	Define
	\begin{equation*}
		\psi_{\underline{n}}(\underline{x},s):=\sum_{\substack{
				\rho:\prod_{p\nmid\infty}\mathbb{Z}_p^*\to G\\
				\underline{n}(\rho)=\underline{n}
		}}\Theta(\rho)^{-s}.
	\end{equation*}
	By similar method as in the proof of Lemma~\ref{lemma:analytic continuation of mu 2}, we know that there exists some regular function \(W(\underline{x},s):U\to\mathbb{C}\) with parameter \(\underline{x}\) such that
	\begin{equation*}
		\begin{aligned}
			\psi_{\underline{n}}(\underline{x},s):=
			&W(\underline{x},s)
			\prod_{\substack{
					\underline{p}\in\mathcal{P}(\lvert G\rvert)^{\gamma}\\
					p_1<\cdots<p_{\gamma}
			}}
		    \sum_{\substack{
			    \rho:\prod_{i=1}^\gamma\mathbb{Z}_{p_i}^*\to G\\
			    \underline{n}(\rho)=\underline{n}
			}}\Theta(\rho)^{-s}\\
			\cdot&\prod_{\substack{
					p\in\mathcal{P}(\lvert G\rvert)\\
					p\nmid p_1\cdots p_{\gamma}
				}}(\sum_{\rho_p:\mathbb{Z}_p^*\to G}(1-\mathbf{1}_{\Omega}(\rho))p^{-\vartheta_{\underline{x}}(\rho)s})
		\end{aligned}
	\end{equation*}	
	The expression $\psi_{\underline{n}}(\underline{x},s):U_0\to\mathbb{C}$ defines a regular function with parameter \(\underline{x}\), using absolute convergence.
	So, we can write it as \(\psi_{\underline{n}}(\underline{x},s)=\sum_da_dd^{-s}\) when \((\underline{x},s)\in U\).
	It is also clear that
	\begin{equation*}
		\psi(\underline{x},s)=\sum_{\underline{n}\text{ admissible}}\psi_{\underline{n}}(\underline{x},s)
	\end{equation*}
	The proof of the analytic properties of $\psi$ is similar to that of $\mu$.
	We only show the part of analytic continuation here.
	If \(\rho_0:\prod_{p\nmid\infty}\mathbb{Z}_p^*\to G\) is a surjective homomorphism such that \(n_i=n_i(\rho_0)\), then we have
	\begin{equation*}
		\begin{aligned}
			\sum_{\substack{
					\underline{p}\in\mathcal{P}(\lvert G\rvert)^{\gamma}\\
					p_1<\cdots<p_{\gamma}
			}}
		    \sum_{\substack{
					\rho:\prod_{i=1}^\gamma\mathbb{Z}_{p_i}^*\to G\\
					\underline{n}(\rho)=\underline{n}
			}}\Theta(\rho)^{-s}&=
		    \sum_{\substack{
			\underline{p}\in\mathcal{P}(\lvert G\rvert)^{\gamma}\\
			p_1\neq\cdots\neq p_{\gamma}
		}}
	        \prod_{k=1}^{\omega}
	        \prod_{i=n_1+\cdots+n_{k-1}+1}^{n_1+\cdots+n_k}\frac{1}{n_k!}
	        (\sum_{\substack{
		    \rho:\mathbb{Z}_{p_i}^*\to G\\
		    \rho(\zeta_{p_i})\sim h_k
		}}1)p^{-\nu(h_k)s}.
		\end{aligned}		
	\end{equation*}
	By similar method as in the proof Lemma~\ref{lemma:analytic continuation of mu 2}, the function
	\begin{equation}\label{eqn:psi 1}
		\begin{aligned}
			&\sum_{\substack{
					\underline{p}\in\mathcal{P}(\lvert G\rvert)^{\gamma}\\
					p_1\neq\cdots\neq p_{\gamma}
			}}\prod_{k=1}^{\omega}\prod_{i=n_1+\cdots+n_{k-1}+1}^{n_1+\cdots+n_k}(\sum_{\substack{
					\rho:\mathbb{Z}_{p_i}^*\to G\\
					\rho(\zeta_{p_i})\sim h_k
			}}1)p^{-\nu(h_k)s}
		    \prod_{\substack{
					p\in\mathcal{P}(\lvert G\rvert)\\
					p\nmid p_1\cdots p_{\gamma}
			}}
		    (\sum_{\rho_p:\mathbb{Z}_p^*\to G}(1-\mathbf{1}_{\Omega}(\rho))p^{-\vartheta_{\underline{x}}(\rho)s})\\
		=&\prod_{p\in\mathcal{P}(\lvert G\rvert)}(1+a(p,\underline{x},s)p^{-xs})
		\sum_{\substack{
				\underline{p}\in\mathcal{P}(\lvert G\rvert)^{\gamma}\\
				p_1\neq\cdots\neq p_{\gamma}
			}}
		\prod_{i=1}^{\gamma}b_i(p_i,\underline{x},s)p_i^{-y_is}
		\end{aligned}
	\end{equation}
	where \(y_i=\nu(h_{k_i})\) with \(k_i=\min\{1\leq k\leq\omega\mid n_1+\cdots+n_{k-1}<i\leq n_1+\cdots+n_k\}\), and \(a\) is a coefficient of main term
	\begin{equation*}
		\sum_{\substack{
				\rho:\mathbb{Z}_p^*\to G\\
				\vartheta_{\underline{x}}(\rho)=x
			}}1-\mathbf{1}_\Omega(\rho),
	\end{equation*}
	while \(b_i\) is a coefficient of main term
	\begin{equation*}
		\sum_{\substack{
				\rho:\mathbb{Z}_p^*\to G\\
				\rho(\zeta_p)\sim h_{k_i}
		}}1.
	\end{equation*}
	Let \(\rho_0:\prod_{p\nmid\infty}\mathbb{Z}_p^*\to G\) be a surjective homomorphism such that \(n_i=n_i(\rho_0)\), then we have 
	\begin{equation*}
		\sum_{\substack{
				1\leq k\leq\omega\\
				\nu(h_k)=y
			}}1=\gamma-\delta(\rho_0).
	\end{equation*}
	In other words, we actually can re-order the indices and write (\ref{eqn:psi 1}) as
	\begin{equation*}
		\begin{aligned}
			&\prod_{p\in\mathcal{P}(\lvert G\rvert)}(1+a(p,\underline{x},s)p^{-xs})
			\sum_{\substack{
					\underline{p}\in\mathcal{P}(\lvert G\rvert)^{\gamma}\\
					p_1\neq\cdots\neq p_{\gamma}
			}}
		\prod_{i=1}^{\gamma}b_i(p_i,\underline{x},s)p_i^{-y_is}\\
		=&\prod_{p\in\mathcal{P}(\lvert G\rvert)}(1+a(p,\underline{x},s)p^{-xs})
		\sum_{\substack{
					\underline{p}\in\mathcal{P}(\lvert G\rvert)^{\gamma}\\
					p_1\neq\cdots\neq p_{\gamma}
			}}
		\prod_{i=1}^{\gamma-\delta(\rho_0)}b_i(p_i,\underline{x},s)p_i^{-ys}
		\prod_{i=\gamma-\delta(\rho_0)+1}^{\gamma}b_i(p_i,\underline{x},s)p_i^{-y_is},
		\end{aligned}
	\end{equation*} 
	where \(y_i>y\) for all \(i>\gamma-\delta(\rho_0)\).
	We can assume without loss of generality that \(\gamma-\delta(\rho_0)>0\).
	Then Proposition~\ref{prop:analytic continuation of zeta(q,a)} and \ref{prop:count (q,a) primes} says that, there exists regular functions $f_\rho,g_{\rho,0},\dots,g_{\rho,\gamma-\delta(\rho)}:\bar{U}_0\to\mathbb{C}$ with parameter such that $g_{\rho,\gamma-\delta(\rho)}$ can be chosen to be a positive constant and $f_\rho(\underline{x},x_0)\neq0$ and such that
	\[\begin{aligned}
		&\prod_{\substack{
				\underline{p}\in\mathcal{P}(\lvert G\rvert)^{\gamma}\\
				p_1\neq\cdots\neq p_{\gamma}
		}}
	    \prod_{k=1}^{\omega}
	    \prod_{i=n_1+\cdots+n_{k-1}+1}^{n_1+\cdots+n_k}
	    (\sum_{\substack{
				\rho:\mathbb{Z}_{p_i}^*\to G\\
				\rho(\zeta_{p_i})\sim h_k
		}}1)p^{-\nu(h_k)s}
	    \prod_{\substack{
				p\in\mathcal{P}(\lvert G\rvert)\\
				p\nmid p_1\cdots p_{\gamma}
		}}
	    (\sum_{\rho_p:\mathbb{Z}_p^*\to G}(1-\mathbf{1}_{\Omega}(\rho))p^{-\vartheta_{\underline{x}}(\rho)s})\\
	=&f(\underline{x},s)(s-x^{-1})^{-\beta}
	\sum_{j=0}^{\gamma-\delta(\rho)}g_{j}(\underline{x},s)\big(\log\frac{1}{s-y^{-1}}\big)^j.
	\end{aligned}\]
	Since $\gamma>\gamma_0$, there exists at least one surjective homomorphism $\chi$ such that \(\delta(\chi)=\delta\) and \(\gamma(\chi)=\gamma\) by Lemma~\ref{lemma:admissible partition}.
	According to Definition~\ref{def:delta0}, we know that \(\delta\leq\gamma_0\), hence \(\gamma-\delta>0\).
	By summing over all admissible partitions, we then get the desired analytic continuation of \(\psi_\gamma(\underline{x},s)\).
\end{proof}
Then let's consider the ``lower bound'' of field counting.
\begin{lemma}\label{lemma:tau 2}
	Let \(\gamma>\gamma_0\) be a fixed integer.
	There exists a generating series \(\tau_{\gamma}(\underline{x},s)=\sum_{d}a_dd^{-s}\) such that \(\tau_{\gamma}(\underline{x},s)\leq\pi_{\gamma}(\underline{x},s)\) and such that it defines a regular function \(U_0\to\mathbb{C}\) with parameter.
	Moreover, there exists regular functions $f,g_0,\dots,g_{\gamma-\delta}:\bar{U}_0\to\mathbb{C}$ with parameter such that $g_{\gamma-\delta}$ can be chosen to be a positive constant and $f_{\underline{x}}(x^{-1})\neq0$, and such that
	\[\tau_{\gamma,\underline{x}}(s)=
	f_{\underline{x}}(s)(s-x^{-1})^{-\beta}
	\sum_{j=0}^{\gamma-\delta}g_{j,\underline{x}}(s)\big(\log\frac{1}{s-y^{-1}}\big)^j.\]
    \end{lemma}
\begin{proof}
	We give a construction of such a series as follows.
	Let $\chi:\prod_{p\nmid\infty}\mathbb{Z}_p^*\to G$ be a surjective homomorphism such that $\delta(\chi)=\delta$, and such that $\gamma(\chi)=\gamma$.
	The existence of such \(\chi\) is given by Lemma~\ref{lemma:admissible partition}.
	Let $\gamma=n_{1}+\cdots+n_{\omega}$ be the partition corresponding to $\chi$, i.e., \((n_1,\dots,n_k)=\underline{n}=\underline{n}(\chi)\).
	Let 
	\begin{equation*}
		T:=\{p\mid\lvert G\rvert\}
		\cup
		\{p\in\mathcal{P}(\lvert G\rvert)\mid\operatorname{im}(\rho\vert_p)\cap\Omega=\emptyset\}.
	\end{equation*}
	Define
	\begin{equation*}
		\tau_{\gamma}(\underline{x},s):=
		\sum_{\rho:\prod_{p\nmid\infty}\mathbb{Z}_p^*\to G}
		\sum_{\substack{
				\underline{n}(\rho)=\underline{n}\\
				\rho\vert_p=\chi\vert_p\,\forall p\in T
			}}
		(\Theta_{\underline{x}}(\rho))^{-s}.
	\end{equation*}
	If \(\rho\) satisfies the conditions that \(\underline{n}(\rho)=\underline{n}\) and that \(\rho\vert_p=\chi\vert_p\) for all \(p\in T\), then we see that \(\rho\) must be surjective with \(\gamma(\rho)=\gamma\).
	So \(\tau_{\gamma}(\underline{x},s)\leq\pi_{\gamma}(\underline{x},s)\) as required.
	As in the proof of Lemma~\ref{lemma:analytic continuation of mu 2}, there exists some regular function \(W(\underline{x},s):U\to\mathbb{C}\) with parameter \(\underline{x}\) such that
	\begin{equation*}
		\tau_{\gamma}(\underline{x},s)=
		W(\underline{x},s)
		\sum_{\substack{
				\underline{p}\in(\mathcal{P}(\lvert G\rvert)\backslash T)^\gamma\\
				p_1<\cdots<p_{\gamma}
			}}
		\sum_{\substack{
			    \rho:\prod_{i=1}^{\gamma}\mathbb{Z}_{p_i}^*\to G\\
			    \underline{n}(\rho)=\underline{n}
			}}(\Theta_{\underline{x}}(\rho))^{-s}
		\prod_{\substack{
			    p\in\mathcal{P}(\lvert G\rvert)\backslash T\\
			    p\nmid p_1\cdots p_{\gamma}
			}}
		(\sum_{\rho:\mathbb{Z}_p^*\to G}(1-\mathbf{1}_\Omega(\rho))p^{-\vartheta_{\underline{x}}s})
	\end{equation*}
	The proof of analytic properties is similar to that of $\psi_{\gamma}(\underline{x},s)$.
	So we just leave it to the reader.
\end{proof}
\subsection{Proof of Theorem~\ref{thm:analytic continuation of pi(s)}}
Let's give the proof when \(\gamma>\gamma_0\).
Other cases are similar.
\begin{proposition}\label{prop:analytic continuation of pi 2}
	Let $\gamma>\gamma_0$ be a fixed positive integer.
	The expression $\pi_{\gamma}(\underline{x},s):U_0\to\mathbb{C}$ defines a regular function with parameter \(\underline{x}\).
	\begin{enumerate}
		\item If $y>x$, then there exists a regular function $f:\bar{U}_0\to\mathbb{C}$ with parameter such that $f(\underline{x},x_1^{-1})\neq0$ and such that
		\[\pi_{\gamma,\underline{x}}(s)=f_{\underline{x}}(s)(s-x^{-1})^{-\beta}.\]
		\item Else if $y\leq x$, then there exists regular functions $f,f_0,\dots,f_{\gamma-\delta}:\bar{U}_0\to\mathbb{C}$ with parameter such that $f_{\underline{x}}(x_1^{-1})f_{\gamma-\delta,\underline{x}}(x_1^{-1})\neq0$ and such that
		\[\pi_{\gamma}(\underline{x},s)=f_{\underline{x}}(s)(s-x^{-1})^{-\beta}\sum_{j=0}^{\gamma-\delta}f_{j,\underline{x}}(s)\big(\log\frac{1}{s-y^{-1}}\big)^j.\]
	\end{enumerate}
\end{proposition}
\begin{proof}
	We start with the case when $G=C_q$, a cyclic group of prime order.
	Since $\bar{G}=\{\{0\},\Omega_1=G\backslash\{0\}\}$, the only choice of a nonempty $\Omega$ is $\Omega=\Omega_1$.
	In other words, for each $K\in\mathcal{S}(G)$, we have $\mathbf{1}_{(\Omega,\gamma)}(K)=1$ if and only if $K$ admits exactly $\gamma$ tamely ramified primes.
	Since
	\[\pi_{\gamma}(\underline{x},s)=\mu_{\gamma}(\underline{x},s)-1,\]
	we see that our conclusion holds by Lemma~\ref{lemma:analytic continuation of mu 2}.
	
	Now let $G$ be a finite abelian group.
	Assume that the conclusion holds for each proper subgroup $H\subseteq G$.
	Let $\Theta_{\underline{x}}:\operatorname{Hom}(\prod_{p\nmid\infty}\mathbb{Z}_p^*,H)$ be the induced invariant, i.e., for each local homomorphism $\rho:\mathbb{Z}_p^*\to H$, let \(\vartheta_{\underline{x}}(\rho):=\vartheta_{\underline{x}}(\imath_H\circ\rho)\), where $\imath_H:H\to G$ is the group embedding.
	And define \(\mathbf{1}_\Omega(\rho):=\mathbf{1}_{\Omega}(\imath_H\circ\rho)\) for all local homomorphisms \(\rho:\mathbb{Z}_p^*\to H\).
	Then we see that
	\[\pi_{\gamma}(\underline{x},s)=\mu_{\gamma}(\underline{x},s)-\sum_{H\subsetneq G}\pi_{H,\gamma}(\underline{x},s).\]
	This shows the existence of analytic continuation of $\pi_{\gamma}(\underline{x},s)$.
	To be precise, there exists some integers $\beta',\gamma'$ and some continuous functions $f,g_0,\dots,g_{\gamma'}:\bar{U}\to\mathbb{C}$ that are regular in the closed half-plane $\sigma\geq x_1^{-1}$ for each $\underline{x}\in V_I$ such that
	\[\pi_{\gamma}(\underline{x},s)=f(\underline{x},s)(s-x^{-1})^{-\beta'}\sum_{j=0}^{\gamma'}g_{j}(s)\big(\log\frac{1}{s-y^{-1}}\big)^j.\]
	It suffices to compute $\beta'$ and $\gamma'$ and to give the pole (limit) behaviour at $s=x_1$.
	
	According to Lemma~\ref{lemma:psi} and \ref{lemma:tau 2}, we see that there exists generating series \(\psi_{\gamma}(\underline{x},s)\) and \(\tau_{\gamma}(\underline{x},s)\) such that
	\begin{equation*}
		\tau_{\gamma}\leq\pi_{\gamma}(\underline{x},s)\leq\psi_{\gamma}(\underline{x},s).
	\end{equation*}
	For each \(\underline{x}\in V\), \(\psi_\gamma(\underline{x},s)\) and \(\tau_\gamma(\underline{x},s)\) admits the following analytic continuation.
	There exists some regular functions $g,g_0,\dots,g_{\gamma-\delta}$ in the closed half-plane $\sigma\geq x_1^{-1}$ such that $g_{\gamma-\delta}$ can be chosen to be a positive constant and $g(\underline{x},x^{-1})\neq0$, and such that
	\[\tau_{\gamma}(\underline{x},s)=g(\underline{x},s)(s-x^{-1})^{-\beta}\sum_{j=0}^{\gamma-\delta}g_{j}(\underline{x},s)(s-y^{-1})^{j}.\]
	On the other hand, there exists regular functions $h,h_0,\dots,h_{\gamma-\delta}$ in the closed half-plane $\sigma\geq x_1^{-1}$ such that $h_{\gamma-\delta}$ can be chosen to be a positive constant and $h(\underline{x},x^{-1})\neq0$ , and such that
	\[\psi_{\gamma}(s)=h(\underline{x},s)(s-x^{-1})^{\beta}\sum_{j=0}^{\gamma-\delta}\big(\log\frac{1}{s-y^{-1}}\big)^j.\]
	Then Theorem~\ref{thm:delange} implies that the required analytic continuation.
	For example, let's say $y>x=x_1$.
	If $\beta'>\beta$, then Theorem~\ref{thm:delange} would imply that
	\[\tilde{N}_{\underline{x},\gamma}(t)=o(N_{\underline{x},\gamma}(t)),\]
	as \(t\to\infty\).
	The contradiction implies $\beta'\geq\beta$.
	If $\beta'<\beta$, then by writing $\pi_{\gamma}(\underline{x},s)=\sum_{d}a_{d}d^{-s}$ and $\psi_{\gamma}(\underline{x},s)=\sum_{d}c_{d}d^{-s}$, we have
	\[\sum_{d<t}c_{d}=o(\sum_{d<t}a_d)\]
	as $t\to\infty$.
	This contradicts Lemma~\ref{lemma:psi}, i.e., $a_d\leq c_d$ for all $d\in S(\underline{x})$.
	So $\beta'\leq\beta$.
	And this proves that when $y>x=x_1$, we have $\beta'=\beta$, hence the analytic continuation of $\pi_{\gamma}(\underline{x},s)$ in this case.
	The proof of other cases are similar.
	And we are done.
\end{proof}
\subsection{Proof of Theorem~\ref{thm:counting abelian fields}}
First assume that \(y=x=x_1\).
Let \(\gamma>\gamma_0\) be an integer.
Let
\begin{equation*}
	\pi_{\gamma,\underline{x}}(s)=f_{\gamma,\underline{x}}(s)(s-x_1)^{-s}\sum_{j=0}^{\gamma-\delta}g_{j,\underline{x}}(s)\Bigl(\log\frac{1}{s-x_1}\Bigr)
\end{equation*}
be the analytic continuation of \(\pi_{\gamma,\underline{x}}(s)\) as in Proposition~\ref{prop:analytic continuation of pi 2}.
The statement of Theorem~\ref{thm:delange} says that
if we write \(\pi_{\gamma,\underline{x}}(s)=\sum_{d\in S(\underline{x})}a_dd^{-s}\), then
\begin{equation*}
	\sum_{\substack{
			d\in S(\underline{x})\\
			d<t
		}}a_d\sim\left\{\begin{aligned}
		&\frac{h_{\underline{x}}(x_1^{-1})}{\Gamma(\beta)}t^{x_1^{-1}}(\log t)^{\beta-1}(\log\log t)^{\gamma-\delta}\quad&\text{if }\beta>0\\
		&(\gamma-\delta)h_{\underline{x}}(x_1^{-1})\frac{t^{x_1^{-1}}}{\log t}(\log\log t)^{\gamma-\delta-1}\quad&\text{if }\beta=0
	\end{aligned}\right.
\end{equation*}
where \(h_{\underline{x}}(x_1^{-1})=f(\underline{x},x_1^{-1})g_{\gamma}(\underline{x},x_1^{-1})\) is a continuous function on \(V\).
This shows the counting result of abelian fields when \(y=x\) and \(\gamma>\gamma_0\).
Other cases are similar and are left to the reader.
\subsection{Distribution of class groups}
Finally, we give a statement to describe the general relation between the ordering of number fields and the distribution of class groups in abelian case.
\begin{theorem}\label{thm:ordering and conjecture}
	Keep the notations \(G,(\mathcal{S},\Theta_{\underline{x}}),\Omega\).
	For a fixed choice \(G/\sim=\bigsqcup_{i=0}^l S_i\) with the open subset \(V=\{\underline{x}\in\mathbb{R}_+^l\mid x_1<\cdots<x_l\}\), the Conjecture~\ref{conj:main body}(\ref{conj:r rami prime less than count fields}) holds if and only if \(y=x_1\) is the smallest among the parameters.
\end{theorem}
\begin{proof}
	First assume that \(y=x=x_1\).
	When \(\gamma>\gamma_0\), we see that
	\begin{equation*}
		R^{\gamma}_{\gamma+1}(\underline{x})=0.		
	\end{equation*}
	In other words, the statement of the Conjecture~\ref{conj:main body}(\ref{conj:r rami prime less than r' rami prime}) holds for all \(\gamma>\gamma_0\).
	If \(\gamma\leq\gamma_{\underline{x}}\), then by Lemma~\ref{lemma:analytic continuation of mu 2}, Theorem~\ref{thm:analytic continuation of pi(s)} and \ref{thm:delange}, we see that
	\begin{equation*}
		\lim_{t\to\infty}\frac{\tilde{N}_{\underline{x},\gamma}(t)}{N_{\underline{x}}(t)}=\lim_{t\to\infty}\frac{c_1(\underline{x})t^{x^{-1}}(\log t)^{\beta_1}(\log\log t)^{\gamma}}{c_2(\underline{x})t^{x^{-1}}(\log t)^{\beta_2}},
	\end{equation*}
	where \(c_1,c_2\) are continuous functions in \(V\), and
	\begin{equation*}
		\begin{aligned}
		    \beta_1=&\beta(\mathfrak{M}),\\
	    	\beta_2=&\beta(\nu^{-1}(x_1)).
		\end{aligned}		
	\end{equation*}
	Since there exists \(h\in\Omega\) such that \(\nu(h)=x_1\), it is clear that \(\beta_2>\beta_1\).
	So
	\begin{equation*}
		\lim_{t\to\infty}\frac{c_1(\underline{x})t^{x^{-1}}(\log t)^{\beta_1}(\log\log t)^{\gamma}}{c_2(\underline{x})t^{x^{-1}}(\log t)^{\beta_2}}=0.
	\end{equation*}
	Therefore, we've shown that the Conjecture~\ref{conj:main body}(\ref{conj:r rami prime less than count fields}) holds for all \(\gamma\geq0\).
	The case when \(x_1=y<x\) is similar to the above one.
	
	On the other hand, if \(y>x=x_1\), then Theorem~\ref{thm:analytic continuation of pi(s)} and \ref{thm:delange} implies that when \(\gamma>\gamma_{\underline{x}}\), we have
	\begin{equation*}
		\lim_{t\to\infty}\frac{N_{\underline{x},\gamma}(t)}{N_{\underline{x}}(t)}=\lim_{x\to\infty}\frac{c_3(\underline{x})t^{x^{-1}}(\log t)^{\beta_2}}{c_2(\underline{x})t^{x^{-1}}(\log t)^{\beta_2}},
	\end{equation*}
	which is a non-zero constant.
	In other words, Conjecture~\ref{conj:main body}(\ref{conj:r rami prime less than count fields}) fails in this case.
	And we are done for the proof.
\end{proof}

\section{\texorpdfstring{$S_3$}{S3}-extensions}\label{section:s3}

In this section, we study the case when $G=S_3$.
Let $\mathcal{S}:=\mathcal{S}(S_3)$, which means that it is the set of non-Galois cubic fields.
In this case, $G_1=\langle(23)\rangle$ is the isotropy subgroup of \(1\), which is a cyclic subgroup of order $2$ in $S_3$.
Let $K_6$ be a Galois $S_3$-field, and let $K_3:=K_6^{G_1}$, and let $K_2$ be the unique quadratic subfield in $K_6$.
Let $\Omega=\{(123),(132)\}\subset S_3$.
For example, $\mathbf{1}_{(\Omega,0)}(K_3)=1$ if $K_3$ has no totally ramified primes.

We first define the ordering of fields in this case.
Recall that Definition~\ref{def:invariant general} gives the idea of setting up an invariant for general number fields.
For each $K\in\mathcal{S}$ and for each pair $(x,y)\in\mathbb{R}^2_+$, we can write
\[\Theta_{(x,y)}(K_3):=d_1^{x}d_2^{y}\]
with $d_1=\operatorname{Disc}(K_2)$ and $d_2$ is the product of totally ramified primes, i.e., $(\operatorname{Nm}\Delta K_6/K_2)^{1/4}$.
Since in this section, we care about the ratio $x:y$ instead of the point $(x,y)$ in \(\mathbb{R}^2_+\), so let's give the following notation of invariant of cubic fields.
\begin{definition}
	For each $x\in\mathbb{R}_+$, and for each Galois \(S_3\)-field \(K_6\), let
	\[\Theta_x(K_6):=\Theta_x(K_3):=d_1d_2^{x}.\]
\end{definition} 
For some examples, $\Theta_2$ is the usual discriminant of cubic fields, and $\Theta_1$ is the product of ramified primes, up to $2$ and $3$.
As our convention, let \(\mathcal{S}(x):=(\mathcal{S},\Theta_x)\) be the set \(\mathcal{S}\) of fields ordered by \(\Theta_x\), and let
\begin{equation*}
	\begin{aligned}
		N_x(t):=&N(\mathcal{S}(x);t)\\
		N_{x,\gamma}(t):=&N(\mathcal{S}(x);(\Omega,\gamma);t)
	\end{aligned}	
\end{equation*}
where \(\gamma\) is a non-negative integer.
Recall that in Definition~\ref{def:function R} we've defined the following expression: for each pair $(\gamma_1,\gamma_2)\in\mathbb{N}^2$, we have
\[R^{\gamma_1}_{\gamma_2}(x):=\lim_{t\to\infty}\frac{N_{x,\gamma_1}(t)}{N_{x,\gamma_2}(t)}.\]
For each \(\gamma\geq0\), we can define the expression
\begin{equation*}
	R^{\gamma}(x):=\lim_{t\to\infty}\frac{N_{x,\gamma}(t)}{N_x(t)}.
\end{equation*}
For example, the result of Davenport and Heilbronn~\cite{Davenport1971Cubic} tells us that the expression \(R^0(2)\) is well-defined with a positive value.
In this section, instead of studying the property of $R$ with a fixed $x$ (or fixed ordering of number fields), we will let $x$ be a variable.
The main result is the following.
\begin{theorem}\label{thm:S3 continuity}
	The function \(R^{\gamma}(x)=0\) for all \(0<x<1\).
	If for each \(\gamma=0,1,2\dots\), the expression \(R^{\gamma}(x)\) is well-defined for all \(x>0\), and \(t(\log\log t)^{\gamma}=o(N_{1}(t))\), then \(R^{\gamma}(x)\) is a continuous function defined on \(x>0\) such that \(R^\gamma(x)=0\) when \(0<x\leq1\) and \(R^{\gamma}(x)neq0\) for all \(x\neq0\).
\end{theorem}
We prove the theorem in several steps.
\subsection{Algebraic theory}
\begin{lemma}\label{lemma:cohomology of C2}
	For each positive integer \(n\), let \(C_n\) denote the cyclic group of order \(n\).
	If \(m\geq3\), and \(C_m\) is a \(C_2\)-module with the nontrivial action \(g_2\cdot g_m=g_m^{-1}\) where \(g_m\), resp. \(g_2\), is a generator of \(C_m\), resp. \(C_2\), then \(H^2(C_2,C_m)=0\).
	In other words, there is only one group \(G=D_m\), the dihedral group, satisfying the short exact sequence
	\begin{equation*}
		1\to C_m\to D_m\to C_2\to1,
	\end{equation*}
	up to isomorphisms between sequences.
\end{lemma}
\begin{proof}
	Since \(C_2\) is a cyclic group, the computation of \(H^2(C_2,C_m)\) reduces to the Tate cohomology \(\hat{H}^0(C_2,C_m)=C_m^{C_2}/N_{C_2}C_m\).
	An element \(g\in C_m\) is fixed by \(g_2\) if and only if \(g=1\), so \(C_m^{C_2}=1\).
	And we get \(H^2(C_2,C_m)=\hat{H}^0(C_2,C_m)=0\).
	Since the zero element in \(H^2(C_2,C_m)=\operatorname{Ext}(C_2,C_m)\) is exactly \(G=C_m\rtimes C_2\), we obtain the short exact sequence
	\begin{equation*}
		1\to C_m\to G\to C_2\to1.
	\end{equation*}
	Then we see that \(G=\langle g_m,g_2\mid g_m^m=g_2^2=1,g_2g_mg_2^{-1}=g_m^{-1}\rangle\), which is exactly the dihedral group \(D_m\).
\end{proof}
\begin{proposition}\label{prop:C2-morphism and Galois S3 field}
	If \(K_2\) is a quadratic number field, \(C_3\) is the non-trivial \(C_2\)-module, then a cubic extension \(K_6/K_2\) give rise to a Galois \(S_3\)-field extension \(K_6/\mathbb{Q}\) if and only if the Artin reciprocity map \(\rho:\operatorname{C}_{K_2}\to C_3\) is a \(C_2\)-morphism.
\end{proposition}
\begin{proof}
	According to Lemma~\ref{lemma:cohomology of C2}, the field extension \(K_6/\mathbb{Q}\) is Galois if and only if the Artin reciprocity map \(\rho:\operatorname{C}_{K_2}\to C_3\) is a \(C_2\)-morphism.
	The Galois group \(G:=\operatorname{Gal}(K_6/\mathbb{Q})\) satisfies the short exact sequence
	\begin{equation*}
		1\to C_3\to G\to C_2\to1.
	\end{equation*}
	If \(C_3\) is a trivial \(C_2\)-module, then clearly \(G\) is isomorphic to \(C_6\).
	Else if \(C_3\) is nontrivial, then according to Lemma, we know that \(G\cong C_3\rtimes C_2\), which is just \(S_3\).
\end{proof}
\subsection{Invariant of cubic fields}
Because of the relation between Artin reciprocity map and cubic field extensions \(K_6/K_2\), we give the following definition.
\begin{definition}\label{def:invariant of C2-morphism}
	Let \(x>0\) be a positive real number, and let \(K_2\) be a quadratic number field.
	If \(\rho:\mathscr{O}_p^*\to C_3\) is a local homomorphism, where \(p\) is a finite rational prime and \(\mathscr{O}_p^*=\prod_{\mathfrak{p}\mid p}\mathscr{O}_{\mathfrak{p}}^*\), then define
	\begin{equation*}
		\vartheta_x(\rho):=\left\{\begin{aligned}
			p^x\quad&\text{if }\rho(\mathscr{O}_p^*)=C_3\\
			1\quad&\text{ otherwise.}
		\end{aligned}\right.
	\end{equation*}
	If \(\rho:J_{K_2}^{S_\infty}\to C_3\) is a homomorphism, then let
	\begin{equation*}
		\Theta_x(\rho):=\prod_{p\nmid\infty}\vartheta_x(\rho\vert_p).
	\end{equation*}
	In particular, if \(\chi:\operatorname{C}_{K_2}\to C_3\) is a lift of \(\rho\), then define \(\Theta_x(\chi):=\Theta_x(\rho)\).
\end{definition}
Given a Galois \(K_6\)-extension such that \(K_6/K_2\) corresponds to an Artin reciprocity map \(\chi:\operatorname{C}_{K_2}\to C_3\), we can see that
\begin{equation*}
	\Theta_x(K_6)=\operatorname{Disc}(K_2)\cdot\Theta_x(\chi).
\end{equation*}
\begin{definition}\label{def:indicator of C2-morphism}
	Let \(K_2\) be a quadratic number field.
	If \(\rho:\mathscr{O}_p^*\to C_3\) is  a local homomorphism, where \(p\) is a finite rational prime, then define
	\begin{equation*}
		\mathbf{1}_\Omega(\rho):=\left\{\begin{aligned}
			1\quad&\text{if }\rho(\mathscr{O}_p^*)=C_3\\
			0\quad&\text{otherwise.}
		\end{aligned}\right.
	\end{equation*}
	For each \(\gamma=0,1,2,\dots\), and for each \(\rho:\prod_{\mathfrak{p}\nmid\infty}\mathscr{O}_{\mathfrak{p}}^*\to C_3\), define
	\begin{equation*}
		\mathbf{1}_{(\Omega,\gamma)}(\rho)=\left\{\begin{aligned}
			1\quad&\text{if there are exactly }\gamma\text{ rational primes }\\
			&p\nmid\infty\text{ such that }\mathbf{1}_\Omega(\rho\vert_p)=1;\\
			0\quad&\text{otherwise.}
		\end{aligned}\right.
	\end{equation*}
	In particular, if \(\chi:\operatorname{C}_{K_2}\to C_3\) is a lift of \(\rho\), then define \(\mathbf{1}_{(\Omega,\gamma)}(\chi):=\mathbf{1}_{(\Omega,\gamma)}(\rho)\).
\end{definition}
Again, one can check that \(\mathbf{1}_{(\Omega,\gamma)}(K_3)=1\) if and only if \(\mathbf{1}_{(\Omega,\gamma)}(\chi)=1\) where \(\chi:\operatorname{C}_{K_2}\to G\) is the Artin reciprocity map corresponding to the abelian cubic extension \(K_6/K_2\).
\subsection{Estimate of field-counting}
For any \(x>0\), let 
\begin{equation*}
	S(x)=\{d>0\mid d=d_1d_2^x\text{ where }(d_1,d_2)\in\mathbb{Z}_+^2\}
\end{equation*}
be the index set.
Let \(y:=\min\{1,x\}\), and let \(U:=\{(x,s)\in\mathbb{R}\times\mathbb{C}\mid x>0,\sigma>0\}\) and let \(U_0:=\{(x,s)\in U\mid\sigma>y^{-1}\}\)
For each \(\gamma=0,1,2,\dots\), define the generating series
\begin{equation*}
	\pi_{\gamma}(x,s):=\sum_{d\in S(x)}\sum_{\substack{
			K\in\mathcal{S}(x)\\
			\Theta_x(K)=d
		}}\mathbf{1}_{(\Omega,\gamma)}(K).
\end{equation*}
\begin{lemma}\label{lemma:S3 lower bound}
	Let \(k:=\mathbb{Q}(\sqrt{-2})\).
	For each \(\gamma>0\), define
	\begin{equation*}
		\tau_{\gamma}(x,s):=(\operatorname{Disc}(k))^{-s}\sum_{\chi:\operatorname{C}_{k}\to C_3}\mathbf{1}_{(\Omega,\gamma)}(\chi)(\Theta_x(\chi))^{-s},
	\end{equation*}
	where \(C_3\) is the non-trivial \(C_2\)-module, and \(\chi\) runs over all \(C_2\)-morphisms.
	Then \(\tau_{\gamma}(x,s)\leq\pi_{\gamma}(x,s)\).
	The expression \(\tau_{\gamma}(x,s)\) defines a continuous function in \(U\) such that for each \(x>0\) the function \(\tau_{\gamma}(x,s)\) is regular in the open half-plane \(\Re(s)>y^{-1}\).
	In particular, when \(\gamma>0\), there exists some continuous functions \(g_0(x,s),\dots,g_{\gamma}(x,s)\) in \(\bar{U}\) such that for each \(x>0\) the functions \(g_0(x,s),\dots,g_{\gamma}(x,s)\) are regular in the closed half-plane \(\Re(s)\geq y\) and such that
	\begin{equation*}
		\tau_{\gamma}(x,s)=\sum_{j=0}^{\gamma}g_{j}(x,s)\Bigl(\log\frac{1}{s-x^{-1}}\Bigr)^{j}.
	\end{equation*}
\end{lemma}
Since \(\operatorname{Cl}_k=1\), there is no unramified extension of \(k\), just like \(\mathbb{Q}\).
So when \(\gamma=0\), the series reduces to \(1\).
\begin{proof}
	It is clear from the expression itself that if we write \(\tau_{\gamma}(x,s)=\sum_{d\in S(x)}a_dd^{-s}\) and \(\pi_{\gamma}(x,s)=\sum_{d\in S(x)}b_dd^{-s}\), then
	\begin{equation*}
		\begin{aligned}
			\sum_{d<t}a_d=&1+\sum_{\substack{
					K_3\in\mathcal{S}\\
					K_2=k
				}}\mathbf{1}_{(\Omega,\gamma)}(K_3)\\
			\leq&1+\sum_{K_3\in\mathcal{S}}\mathbf{1}_{(\Omega,\gamma)}(K_3).
		\end{aligned}
	\end{equation*}
	If \(\rho:\prod_{\mathfrak{p}\nmid\infty}\mathscr{O}_{\mathfrak{p}}^*\to C_3\) is any homomorphism, then it is clear that \(\rho(-1)=1\), because
	\begin{equation*}
		1=\rho(1)=\rho(-1)^2.
	\end{equation*}
	Therefore, there exists a unique lift \(\chi:\operatorname{C}_k\to C_3\) of \(\rho\).
	To be precise, we first lift \(\rho\) to the morphism \(\prod_{\mathfrak{p}}\mathscr{O}_{\mathfrak{p}}^*\to C_3\) by adding the trivial map at infinity.
	Then using the fact that \(\rho\) factors through \(k^{S_\infty}\), we obtain a map \(\tilde{\rho}:J_k^{S_\infty}/k^{S_\infty}\to C_3\).
	Since \(\operatorname{Cl}_k=1\), we see that \(J_k^{S_\infty}/k^{S_\infty}\cong\operatorname{C}_k\), hence a homomorphism \(\chi:\operatorname{C}_k\to C_3\).
	Moreover, if \(\rho\) is a \(C_2\)-morphism, then clearly \(\tilde{\rho}\) also respects the \(C_2\)-action.
	The isomorphism \(J_k^{S_\infty}/k^{S_\infty}\cong\operatorname{C}_k\) also respects \(C_2\)-action.
	Therefore if \(\rho\) is a \(C_2\)-morphism, then its lift \(\chi\) is also a \(C_2\)-morphism.
	So, instead of considering the analytic properties of \(\tau_{\gamma}\), let's consider the following generating series
	\begin{equation*}
		\tau'_{\gamma}(x,s):=\sum_{\rho}\mathbf{1}_{(\Omega,\gamma)}(\rho)(\Theta_x(\rho))^{-s},
	\end{equation*}
	where \(C_3\) is the non-trivial \(C_2\)-module and \(\rho\) runs over all \(C_2\)-morphisms \(\prod_{\mathfrak{p}\nmid\infty}\mathscr{O}_{\mathfrak{p}}^*\to C_3\).
	When \(\gamma>0\), the function \(\tau'_{\gamma}(x,s)\) could be re-written in the realm of absolute convergence:
	\begin{equation*}
		\tau'_{\gamma}(x,s)=\sum_{\substack{
				\underline{p}\in\mathcal{P}\\
				3<p_1<\cdots<p_{\gamma}
			}}\prod_{i=1}^{\gamma}(\sum_{\rho:\mathscr{O}_{p_i}^*\to C_3}p_i^{-xs}),
	\end{equation*}
	where \(\rho\) runs over all surjective \(C_2\)-morphisms \(\mathscr{O}_{p_i}^*\to C_3\).
	By considering the case when \(x=1\), and comparing with Riemann zeta function, we know that \(\tau'_{\gamma}(x,s)\) as a series is absolutely convergent in \(U\), hence continuous in \(U\) and \(\tau'_{\gamma}(x,s)\) is a regular function in the open half-plane \(\Re(s)>x^{-1}\) for each fixed \(x>0\).
	For each rational prime \(p>3\), the number of surjective \(C_2\)-morphisms is totally determined by its class modulo \(48\).
	To be precise, if \(p\equiv5,7,13,15\bmod{16}\), then \(p\) is split in \(k\), and \(p\) is inert otherwise.
	On the other hand, let \(\mathfrak{P}\mid p\) be a prime of \(K_6\) lying above \(p\), and let \(G(\mathfrak{P})\subseteq S_3\) be the decomposition group of \(\mathfrak{P}\).
	If \(y\) is the generator of the inertia subgroup, and \(x\) is the Frobenius element, up to conjugacy class, then we have the following identity
	\begin{equation*}
		xyx^{-1}=y^p.
	\end{equation*}
	In particular, this implies that if \(\langle y\rangle=C_3\), then either \(p\equiv-1\bmod{3}\) and \(p\) is inert in \(k\), or \(p\equiv1\bmod{3}\) and \(p\) is split in \(k\).
	Since \(3\) is coprime to \(16\), we see that Chinese Remainder Theorem guarantees solutions in \(\mathbb{Z}/48\).
	So by Proposition~\ref{prop:analytic continuation of zeta(q,a)}, we know that there exists some continuous functions \(g_0(x,s),\dots,g_{\gamma}(x,s)\) in \(\bar{U}\) such that for each \(x>0\) the functions \(g_0(x,s),\dots,g_{\gamma}(x,s)\) are regular in the closed half-plane \(\Re(s)\geq y\) and such that
	\begin{equation*}
		\tau_{\gamma}(x,s)=\sum_{j=0}^{\gamma}g_j(x,s)\Bigl(\log\frac{1}{s-x^{-1}}\Bigr)^{j}.
	\end{equation*}
\end{proof}
\begin{lemma}\label{lemma:S3 upper bound}
	For each \(\gamma>0\), let
	\begin{equation*}
		\mu_{\gamma}(x,s):=
		\sum_{K_2\in\mathcal{S}(C_2)}\lvert\operatorname{Hom}(\operatorname{Cl}_{K_2},C_3)\rvert(\operatorname{Disc}(K_2))^{-s}
		\sum_{\rho:J_{K_2}^{S_\infty}\to C_3}\mathbf{1}_{(\Omega,\gamma)}(\rho)(\Theta_x(\rho))^{-s},
	\end{equation*}
	where \(C_3\) is the non-trivial \(C_2\)-module and \(\rho\) runs over all \(C_2\)-morphisms.
	Then \(\mu_{\gamma}(x,s)\geq\pi_{\gamma}(x,s)\).
\end{lemma}
\begin{proof}
	According to Lemma~\ref{prop:CFT obstruction}, for a fixed quadratic number field \(K_2\) and a homomorphism \(\rho:J_{K_2}^{S_\infty}\to C_3\), there exists at most \(\lvert\operatorname{Hom}(\operatorname{Cl}_{K_2},C_3)\rvert\) lift \(\chi:\operatorname{C}_{K_2}\to C_3\) of \(\rho\).
	For an Artin reciprocity map \(\chi\in\operatorname{Hom}_{C_2}(\operatorname{C}_{K_2},C_3)\), we have that \(\mathbf{1}_{(\Omega,\gamma)}(\chi)=1\) if and only if \(\mathbf{1}_{(\Omega,\gamma)}(K_3)=1\).
	Therefore, if we write \(\mu_{\gamma}(x,s)=\sum_{d\in S(x)}a_dd^{-s}\) and \(\pi_{\gamma}(x,s)=\sum_{d\in S(x)}b_dd^{-s}\), then for a fixed \(d_0\in S(x)\), we have
	\begin{equation*}
		\begin{aligned}
			a_{d_0}&=\sum_{\substack{
					K_2=\mathbb{Q}(\sqrt{D})\\
					\operatorname{Disc}(K_2)<d_0
				}}\lvert\operatorname{Hom}(\operatorname{Cl}_{K_2},C_3)\rvert\sum_{\substack{
				\rho:J_{K_2}^{S_\infty}\to C_3\\
				\Theta_x(\rho)=d_0/\operatorname{Disc}(K_2)
			}}\mathbf{1}_{(\Omega,\gamma)}(\rho)\\
		\geq&\sum_{\substack{
				K_2=\mathbb{Q}(\sqrt{D})\\
				\operatorname{Disc}(K_2)<d_0
		}}\sum_{\substack{
				\chi:C_{K_2}\to C_3\\
				\Theta_x(\chi)=d_0/\operatorname{Disc}(K_2)
		}}\mathbf{1}_{(\Omega,\gamma)}(\chi)\\
	=&\sum_{\substack{
		K_3\in\mathcal{S}\\
		\Theta_x(K_3)=d_0
	}}\mathbf{1}_{(\Omega,\gamma)}(K_3)=b_{d_0}.
		\end{aligned}
	\end{equation*}
	And we are done.
\end{proof}
\begin{proposition}\label{prop:S3 estimate 0<x<1}
	If \(0<x<1\), and \(\gamma>0\) is an integer, then we have
	\begin{equation*}
		N_{x,\gamma}(t)\asymp\frac{t^{x^{-1}}}{\log t}(\log\log t)^{\gamma},
	\end{equation*}
	when \(t>1+e\).
\end{proposition}
\begin{proof}
	According to Lemma~\ref{lemma:S3 lower bound} and Theorem~\ref{thm:delange}, we see that
	\begin{equation*}
		N_{x,\gamma}(t)\gg\frac{t^{x^{-1}}}{\log t}(\log\log t)^{\gamma}.
	\end{equation*}
	Let \(K_2\) be any quadratic number field.
	Then
	\begin{equation*}\label{eqn:S3 estimate 1}
		\sum_{\rho:J_{K_2}^{S_\infty}\to C_3}\mathbf{1}_{(\Omega,\gamma)}(\rho)(\Theta_x(\rho))^{-s}\leq
		2^{\gamma+1}\sum_{\substack{
				\underline{p}\in\mathcal{P}^{\gamma}\\
				3<p_1<\cdots<p_{\gamma}
		}}(p_1\cdots p_{\gamma})^{-xs}.
	\end{equation*}
	Because \(\operatorname{Hom}_{C_2}(\mathscr{O}_p^*,C_3)\) has at most three elements when \(p\nmid6\infty\).
	On the other hand, \(\operatorname{Hom}_{C_2}(K_{2,\infty}^*,C_3)\) contains at most \(2\) elements, i.e., it is non-trivial when \(K_2/\mathbb{Q}\) is real and \(K_{2,\infty}=\mathbb{R}_{v_1}\times\mathbb{R}_{v_2}\) with \(v_1,v_2\) the two valuations above \(\infty\) and the Galois action of \(C_2\).
	By Proposition~\ref{prop:count (q,a) primes}, we see that there exists some continuous functions \(h_{0}(x,s),\dots,h_{\gamma}(x,s)\), where \(h_\gamma\) is a non-zero constant, on \(\bar{U}\), such that for each \(x>0\), \(h_0(x,s),\dots,h_{\gamma}(x,s)\) are regular in the closed half-plane \(\Re(s)\geq x^{-1}\) and such that
	\begin{equation*}
		\sum_{\substack{
				\underline{p}\in\mathcal{P}^{\gamma}\\
				3<p_1<\cdots<p_{\gamma}
		}}(p_1\cdots p_{\gamma})^{-xs}=\sum_{j=0}^{\gamma}h_j(x,s)\Bigl(\log\frac{1}{s-x^{-1}}\Bigr)^j.
	\end{equation*}
	Let's we write
	\begin{equation*}
		\sum_{d\in S(x)}c_dd^{-s}:=2^{\gamma+1}\sum_{\substack{
				\underline{p}\in\mathcal{P}^{\gamma}\\
				3<p_1<\cdots<p_{\gamma}
		}}(p_1\cdots p_{\gamma})^{-xs}.
	\end{equation*} 
	Then there exists some constant \(c_1>0\) such that for all quadratic number field \(K_2\) and for all \(t>0\) we have
	\begin{equation}
		\begin{aligned}
			\sum_{\substack{
					K_3\in\mathcal{S},\Theta_x(K_3)<t\\
					K_2\subseteq K_6
			}}\mathbf{1}_{(\Omega,\gamma)}(K_3)\leq&\lvert\operatorname{Hom}(\operatorname{Cl}_{K_2},C_3)\rvert\sum_{d<t}c_d\\
			\leq&\lvert\operatorname{Hom}(\operatorname{Cl}_{K_2},C_3)\rvert\cdot\frac{c_1t^{x^{-1}}}{\log t}(\log\log t)^{\gamma},
		\end{aligned}		
	\end{equation}
	where \(t>1+e\) the constant \(c_1\) does not depend on \(K_2\).
	Write \(\mu_{\gamma}(x,s):=\sum_{d\in S(x)}a_dd^{-s}\), then
	\begin{equation}\label{eqn:S3 estimate 2}
		\begin{aligned}
			N_{x,\gamma}(t)\leq&\sum_{d<t}a_d\\
			=&\sum_{\substack{
					K_2=\mathbb{Q}(\sqrt{D})\\
					\operatorname{Disc}(K_2)<t
				}}\lvert\operatorname{Hom}(\operatorname{Cl}_{K_2},C_3)\rvert
			\sum_{\substack{
					\rho:J_{K_2}^{S_\infty}\to C_3\\
					\Theta_x(\rho)<t/\operatorname{Disc}(K_2)
			}}\mathbf{1}_{(\Omega,\gamma)}(\rho)\\
		\leq&\sum_{\substack{
				d\in S(x)\\
				d<t
			}}c_d\sum_{\substack{
				K_2=\mathbb{Q}(\sqrt{D})\\
				\operatorname{Disc}(K_2)<t/d
		}}\lvert\operatorname{Hom}(\operatorname{Cl}_{K_2},C_3)\rvert
		\end{aligned}
	\end{equation}
	According to Davenport and Heilbronn~\cite[Theorem 3]{Davenport1971Cubic}, for all \(t>0\), there exists some constant \(c_2>0\) such that
	\begin{equation*}
		\sum_{\substack{
				K_2=\mathbb{Q}(\sqrt{D})\\
				\operatorname{Disc}(K_2)<t
		}}\lvert\operatorname{Hom}(\operatorname{Cl}_{K_2},C_3)\rvert=\sum_{\substack{
		K_2=\mathbb{Q}(\sqrt{D})\\
		\operatorname{Disc}(K_2)<t
	}}\lvert\operatorname{Cl}_{K_2}[3]\rvert\leq c_2t.
	\end{equation*}
	Therefore, (\ref{eqn:S3 estimate 2}) gives
	\begin{equation}\label{eqn:S3 estimate 3}
		\begin{aligned}
			\sum_{\substack{
					d\in S(x)\\
					d<t
			}}c_d\sum_{\substack{
					K_2=\mathbb{Q}(\sqrt{D})\\
					\operatorname{Disc}(K_2)<t/d
			}}\lvert\operatorname{Hom}(\operatorname{Cl}_{K_2},C_3)\rvert=&\sum_{\substack{
					d\in S(x)\\
					d<t
			}}c_d\sum_{\substack{
			K_2=\mathbb{Q}(\sqrt{D})\\
			\operatorname{Disc}(K_2)<t/d
		}}\lvert\operatorname{Cl}_{K_2}[3]\rvert\\
	&\leq c_2\sum_{\substack{
			d\in S(x)\\
			d<t
	}}c_d\frac{t}{d}
		\end{aligned}
	\end{equation}
	Let \(C(t):=\sum_{d<t}c_d\).
	Note that actually \(C(t)=\sum_{2^{x}\leq d<t}c_d\).
	Using Riemann-Stieltjes integral and integration by parts (see Montgomery and Vaughan~\cite[Appendix A]{montgomery2006multiplicative}), (\ref{eqn:S3 estimate 3}) shows that
	\begin{equation}\label{eqn:S3 estimate 4}
		\begin{aligned}
			\sum_{\substack{
					d\in S(x)\\
					d<t
			}}c_d\frac{t}{d}=&\int_{1^{+}}^t\frac{t}{y}\operatorname{d}C(y)\\
		    =&C(y)\frac{t}{y}\Big\vert_{1^+}^t-\int_{1^+}^tC(y)\operatorname{d}\frac{t}{y}\\
		    =&C(t)+t\int_{1^+}^tC(y)y^{-2}\operatorname{d}y.
		\end{aligned}		
	\end{equation}
	Using the idea that \(\int_a^b f(x)\operatorname{d}x\leq(b-a)\max_{[a,b]}f(x)\), we see that
	\begin{equation*}
		\begin{aligned}
			\int_{1^+}^tC(y)y^{-2}\operatorname{d}y\leq&\int_{1^+}^{1+e}C(y)y^{-2}\operatorname{d}y
			+c_1\int_{1+e}^{t}\frac{y^{x^{-1}-2}}{\log y}(\log\log y)^{\gamma}\operatorname{d}y&&\\
			\ll&\left\{\begin{aligned}
				&t^{x^{-1}-1}(\log t)^{-1}(\log\log t)^{\gamma}\quad&\text{if }x^{-1}-2>0\\
				&t(\log\log t)^{\gamma}\quad&\text{otherwise.}
			\end{aligned}
			\right.
		\end{aligned}
	\end{equation*}
	Either case, we see that 
	\begin{equation*}
		t\int_{1^+}^tC(y)y^{-2}\operatorname{d}y\ll\frac{t^{x^{-1}}}{\log t}(\log\log t)^{\gamma}.
	\end{equation*}
	Put the result of (\ref{eqn:S3 estimate 4}) back into (\ref{eqn:S3 estimate 2}), we have
	\begin{equation*}
		N_{x,\gamma}(t)\ll\frac{t^{x^{-1}}}{\log t}(\log\log t)^{\gamma}.
	\end{equation*}
	And we are done.
\end{proof}
Finally we need to understand what happens when $t>1$ with the condition that $\alpha$ is a function defined in some neighbourhood of a point.
\begin{lemma}\label{lemma:S3 estimate x>1}
	\begin{enumerate}
		\item When \(x>1\), and \(\gamma>0\), we have
		\begin{equation*}
			N_{x,\gamma}(t)\ll t
		\end{equation*}
		as \(t>1+e\).
		\item When \(x=1\), and \(\gamma>0\), we have
		\begin{equation*}
			N_{x,\gamma}(t)\ll t(\log\log t)^{\gamma+1}
		\end{equation*}
		as \(t>1+e\).
	\end{enumerate}	
\end{lemma}
\begin{proof}
	Lemma~\ref{lemma:S3 upper bound} and (\ref{eqn:S3 estimate 1}) are independent of the choice of \(x\).
	So the integral of (\ref{eqn:S3 estimate 4}) in this case says the following.
	\begin{equation*}
		\int_{1^+}^tC(y)y^{-2}\operatorname{d}y\leq\int_{1^+}^{1+e}C(y)y^{-2}\operatorname{d}y+c_1\int_{1+e}^t\frac{y^{x^{-1}-2}}{\log y}(\log\log y)\operatorname{d}y.
	\end{equation*}
	If \(x>1\), then \((\log t)^{\gamma}/t=O(1)\) in the region \(t\geq1\), and we see that the Laplace transform
	\begin{equation*}
		\int_1^\infty e^{-at}\frac{(\log t)^{\gamma}}{t}\operatorname{d}t
	\end{equation*}
	is well-defined for all \(a>0\).
	So we have
	\begin{equation*}
		\int_{1+e}^t\frac{y^{x^{-1}-2}}{\log y}(\log\log y)\operatorname{d}y\stackrel{y=e^t}{=}\int_{\log(1+e)}^{\infty}e^{-(1-x^{-1})t}\frac{(\log t)^{\gamma}}{t}\operatorname{d}t<\infty,
	\end{equation*}
	for all \(x>1\).
	By putting the results together back to (\ref{eqn:S3 estimate 2}) gives the estimate:
	\begin{equation*}
		N_{x,\gamma}(t)\ll t.
	\end{equation*}
	Else if \(x=1\), then the integral of (\ref{eqn:S3 estimate 4}) becomes
	\begin{equation*}
		\begin{aligned}
			\int_{1^+}^{t}C(y)y^{-2}\operatorname{d}y\leq&\int_{1^+}^{1+e}C(y)y^{-2}\operatorname{d}y+c_2\int_{1+e}^t\frac{(\log\log y)^{\gamma}}{y\log y}\operatorname{d}y\\
			\ll&(\log\log t)^{\gamma+1}.
		\end{aligned}
	\end{equation*}
	Therefore, putting the results back to (\ref{eqn:S3 estimate 2}) we have
	\begin{equation*}
		N_{x,\gamma}(t)\ll t(\log\log t)^{\gamma+1}.
	\end{equation*}
\end{proof}
\begin{proposition}\label{prop:S3 continuity}
	Let $x_0>1$ be a real number, and let $(\gamma_1,\gamma_2)$ be a pair of natural numbers.
	\begin{enumerate}
		\item If the functions $R^0_{\gamma_1}(x)$ and $R^{\gamma_1}_{\gamma_2}(x)$ are both well-defined in a neighbourhood of $x=x_0$, then $R^{\gamma_1}_{\gamma_2}(x_0)>0$, and $R^{\gamma_1}_{\gamma_2}(x)$ is continuous at $x_0$.
		\item If the function $R^0(x)$ is well-defined in a neighbourhood of \(x_0\), then \(R^0(x)\) is continuous in a neighbourhood of \(x=x_0\) such that \(R^0(x)\neq0\) for all \(x\) near \(x_0\).
	\end{enumerate}	
\end{proposition}
\begin{proof}
	By ``well-defined'', we mean that the image $R^{\gamma_1}_{\gamma_2}(x)$ is a real number.
	According to \cite[Theorem 8]{bhargava2013davenport}, we have
	\[N_{2,0}(t)\sim c_0t\]
	where $c_0>0$ is a constant.
	This shows that $R^0_{\gamma_1}(t)$, if well-defined in a neighbourhood of $t_0$, must be positive.
	Otherwise, this would imply that
	\[c_0t\sim N_{2,0}(t)=N_{x_0,0}(t)=o(N_{x_0,\gamma_1}(t)),\]
	which contradicts our estimates of counting fields in Lemma~\ref{lemma:S3 estimate x>1}.
	So, there exist some $c_1>0$ such that
	\[N_{x_0,\gamma_1}(t)\sim c_1t.\]
	By similar arguments, we know that there exists some constant $c_2$ such that
	\[N_{x_0,\gamma_2}(t)\sim c_2t.\]
	Of course, this implies that $R^{\gamma_1}_{\gamma_2}(x_0)=c_1/c_{2}>0$.
	Now, for a small $\delta>0$, and $j=1,2$, we compare $N_{x_0,\gamma_j}(t)$ and $N_{x_0+\delta,\gamma_j}(t)$.
	First note that for all $t>0$ we have
	\begin{equation*}
		\begin{aligned}
			\{d\in S(x)\mid d<t\}=&\{(d_1,d_2)\in\mathbb{N}^2\mid d_1<t,d_2^x>\frac{t}{d_1}\}\\
			\supseteq&\{(d_1,d_2)\in\mathbb{N}^2\mid d_1<t,d_2^{x+\delta}<\frac{t}{d_1}\}\\
			=&\{d\in S(x+\delta)\mid d<t\}.
		\end{aligned}
	\end{equation*}
	This already shows that
	\[N_{x_0,\gamma_j}(t)\geq N_{x_0+\delta,\gamma_j}(t).\]
	Define the following function
	\[D_{\delta,\gamma_j}(t):=\sum_{\substack{
			d_1d_2^x<t\\
			\omega(d_2)=\gamma_j
	}}1-\sum_{\substack{
			d_1d_2^{x+\delta}<t\\
			\omega(d_2)=\gamma_j
	}}1,\]
    where \(\omega(n)\) is the number of different prime factors of \(n\).
	For each $d=d_1d_2^x$, resp. $d=d_1d_2^{x+\delta}$, there are at most $2^{\gamma_j}$ many different cubic fields $K_3\in\mathcal{S}$ with $\Theta_x(K_3)=d$, resp. $\Theta_{x+\delta}(K_3)=d$.
	Because if there exists at least one \(K_3\) such that \(\Theta_x(K_3)=d\), then $d_1$ is the discriminant of the associated quadratic number field $K_2$, and $d_2$ becomes the product of totally ramified primes in \(K_3/\mathbb{Q}\) which corresponds to the primes that are ramified in \(K_6/K_2\).
	Since there are at most $2^{\gamma_j}$ different \(S_3\)-fields $K_6$ lying above $K_2$, we obtain the upper bound for the number of \(K_3\) such that \(\Theta_x(K_3)=d=d_1d_2^x\).
	So, 
	\[N_{x_0,\gamma_j}(t)-N_{x_0+\delta,\gamma_j}(t)\leq2^{\gamma_j}D_{\delta,\gamma_j}(x),\]
	or we could say 
	\[N_{x_0,\gamma_j}(t)-2^{\gamma_j}D_{\delta,\gamma_j}(x)\leq N_{x_0+\delta,\gamma_j}(t).\]
	According to Proposition~\ref{prop:count (q,a) primes} and Theorem~\ref{thm:delange}, we know that there exists some regular function $g_j(s)$ in the closed half-plane $\sigma\geq1$ such that
	\[\begin{aligned}
		N_{x_0+\delta,\gamma_j}(t)\geq&N_{x_0,\gamma_j}(t)-2^{\gamma_j}D_{\delta,\gamma_j}(t)\\
		\sim&\left(c_j+(g_j(x_0)-g_j(x_0+\delta))\right)t.
	\end{aligned}\]
	Finally, we can give the following estimate of $R^{\gamma_1}_{\gamma_2}(x_0+\delta)$.	
	Since
	\[\begin{aligned}
		\frac{N_{x_0,\gamma_1}(t)-2^{\gamma_1}D_{\delta,\gamma_1}(t)}{N_{x_0,\gamma_2}(t)}&\leq\frac{N_{x_0+\delta,\gamma_1}(t)}{N_{x_0+\delta,\gamma_2}(t)}\\
		&\leq\frac{N_{x_0,\gamma_1}(t)}{N_{x_0,\gamma_2}(t)-2^{\gamma_2}D_{\delta,\gamma_2}(x)},
	\end{aligned}\]
	as \(t\to\infty\), we have
	\[\frac{c_1+(g_1(x)-g_1(x_0+\delta))}{c_2}\leq R^{\gamma_1}_{\gamma_2}(x_0+\delta)\leq\frac{c_1}{c_{2}+(g_2(x)-g_{2}(x+\delta))}.\]
	Since the functions $g_1,g_{2}$ are both regular in the closed half-plane $\Re(s)\geq1$, we know that for each $\epsilon>0$, when $\delta$ is small enough, $\lvert R^{\gamma_1}_{\gamma_2}(x_0+\delta)-R^{\gamma_1}_{\gamma_2}(x_0)\rvert<\epsilon$ holds.
	Similar argument works for $R^{\gamma_1}_{\gamma_2}(x_0-\delta)$.
	And this shows that $R^{\gamma_1}_{\gamma_2}(x)$ is continuous at $x_0$, as required.
	
	Now assume that \(R^0(x)\) is well-defined in a neighbourhood of \(x_0\).
	Then by Lemma~\ref{lemma:S3 estimate x>1}, we know that  \(R^0(x)\) must be non-zero near \(x_0\), i.e., \(N_{x}(t)\sim c(x)t\) near \(x_0\), where \(c(x)\) is a positive number dependent on \(x\).
	Since for a fixed \(t>0\) and for all \(x,x'>0\) we have \(N_{x,0}(t)=N_{x',0}(t)\), the continuity of \(R^0(x)\) reduces to the continuity of the function \(\lim_{t\to\infty}t^{-1}N_{x}(t)\).
	Let \(\delta\) be a small positive number.
	Define
	\begin{equation*}
		D_{\delta}(t):=\sum_{j=1}^\infty 2^jD_{\delta,j}(t).
	\end{equation*}
	For each \(t>0\), the expression \(D_{\delta}(t)\) is actually a finite sum.
	Again
	\begin{equation*}
		N_{x_0}(t)-N_{x_0+\delta}(t)\leq D_{\delta}(t).
	\end{equation*}
	Moreover, we can compute \(D(t)\) using the generating series
	\begin{equation*}
		\mathfrak{D}(s):=\sum_{K_2=\mathbb{Q}(\sqrt{D})}\operatorname{Disc}(K_2)^{-s}(\prod_{p\nmid\infty}(1+2p^{-x_0s})-\prod_{p\nmid\infty}(1+2p^{-(x_0+\delta)s})).
	\end{equation*}
	In other words, there exists some regular function \(h(s)\) in the closed half-plane \(\Re(s)\geq1\) such that
	\begin{equation*}
		D_{\epsilon}(t)\sim h(x_0+\delta)t.
	\end{equation*}
	Then we see that
	\begin{equation*}
		N_{x_0}(t)-N_{x_0+\delta}(t)\leq D_{\delta}(t)\sim h(x_0+\delta)t
	\end{equation*}
	which implies that for all \(\epsilon>0\), take \(\delta\) to be a small enough number, we have that 
	\begin{equation*}
		\lim_{t\to\infty}t^{-1}(N_{x_0}(t)-N_{x_0+\delta}(t))=h(x_0+\delta)<\epsilon
	\end{equation*}
	Similar result holds for \(N_{x_0-\delta}(t)\).
	And we are done.
\end{proof}
\subsection{Proof of Theorem~\ref{thm:S3 continuity}}
When \(0<x<1\), using Proposition~\ref{prop:S3 estimate 0<x<1}, for all \(0<\gamma_1<\gamma_2\) we have
\begin{equation*}
	\begin{aligned}
		R^{\gamma_1}_{\gamma_2}(x)=&\lim_{t\to\infty}\frac{N_{x,\gamma_1}(t)}{N_{x,\gamma_2}(t)}\\
		\ll&\lim_{t\to\infty}\frac{t^{x^{-1}}(\log t)^{-1}(\log\log t)^{\gamma_1}}{t^{x^{-1}}(\log t)^{-1}(\log\log t)^{\gamma_2}}=0.
	\end{aligned}
\end{equation*}
Therefore \(R^{\gamma}(x)<R^{\gamma}_{\gamma+1}(x)=0\).
When \(\gamma=0\), we have \(N_{x,0}(t)\sim c_0x=o(N_{x,1}(t))\) with \(c_0\) a non-zero constant.
So \(R^{0}(x)\) is also \(0\).

If \(t(\log\log t)^{\gamma}=o(N_{1}(t))\), then Lemma~\ref{lemma:S3 estimate x>1} implies that
\begin{equation*}
	\begin{aligned}
		R^{\gamma}(1)=&\lim_{t\to\infty}\frac{N_{1,\gamma}(t)}{N_1(t)}\\
		\leq&\lim_{t\to\infty}\frac{c_{\gamma}t(\log\log t)^{\gamma}}{N_1(t)}=0.
	\end{aligned}	
\end{equation*}

If for a fixed \(\gamma\geq0\), the functions \(R^0(x)\) and \(R^\gamma(x)\) are well-defined for all \(x>0\), then first of all Proposition~\ref{prop:S3 continuity} shows that \(R^0(x)\) is continuous and \(R^0(x)\) is non-zero for all \(x>1\).
By the identity
\begin{equation*}
	R^\gamma(x)=R_0^\gamma(x)R^0(x),
\end{equation*}
we see that \(R_0^\gamma(x)\) is well-defined when \(x>1\).
So, Proposition~\ref{prop:S3 continuity} shows that \(R^\gamma_0(x)\) is also non-zero and continuous when \(x>1\), hence \(R^\gamma(x)\) is also continuous and non-zero for all \(\gamma>0\).

\begin{appendices}
	\section{Delange's Theorem}
	For the convenience of the reader, we cite some results of Delange~\cite{Delange54} here.
	Let \(A(t)\) be a real function defined for \(t\geq0\), non-decreasing and non-negative.
	Assume that the integral \(\int_0^\infty e^{-st}A(t)\operatorname{d}t\) is convergent in the half-plane \(\Re(s)>\alpha\), where \(\alpha\) is a positive real number, and we define its value by \(f(s)\).
	The way we state the result is slightly different from Delange's original one.
	
	When we say \(G\) is an analytic continuation of a holomorphic function \(F\) defined on some subsets \(V\) of \(\mathbb{C}\), it means that there exists a holomorphic function \(G(s)\) defined on a subset \(U\) containing \(V\) such that
	\begin{equation*}
		F(s)=G(s)
	\end{equation*}
	for all \(s\in V\).
	If \(U,V\) are connected, then we know that such an analytic continuation \(G\) is unique.
	In this case, we may as well just say that \(F\) is holomorphic on \(U\), and for \(s\in U\backslash V\), its value is given by \(G(s)\).
	
	For example, the Riemann zeta function defined by the series \(\zeta(s)=\sum n^{-s}\) is not well-defined when it comes to \(\Re(s)\leq1\), and in this case we should use its analytic continuation given by the functional equation.
	However, since the analytic continuation of \(\zeta(s)\) is unique, it is common that we still adopt the notation \(\zeta(s)\) and say that \(\zeta(s)\) is a meromorphic function over the whole complex plane.
	Now let's present the results.
	\begin{theorem}\label{Theorem III}\cite[Th\'{e}or\`{e}me III]{Delange54}
		Let \(\beta\) be a real number that is not an integer\(\leq0\).
		If there exists holomorphic functions \(g\) and \(h\) in a neighbourhood of \(\Re(s)\geq\alpha\), and \(g(\alpha)\neq0\), such that
		\begin{equation*}
			f(s)=(s-\alpha)^{-\beta}g(s)+h(s),
		\end{equation*}
		for all \(\Re(s)>\alpha\),
		then as \(t\to\infty\), we have that
		\begin{equation*}
			A(t)\sim\frac{g(\alpha)}{\Gamma(\beta)}e^{\alpha t}t^{\beta-1}.
		\end{equation*}
	\end{theorem}
	\begin{remark}
		\begin{enumerate}
			\item There is another case in the original statement, which could be thought of as treating functions with more subtle analytic properties at \(s=\alpha\).
			But we do not need it here.
			\item The condition essentially says that \(f\) admits an analytic continuation to a neighbourhood of \(\Re(s)\geq\alpha\) that satisfies a certain form.
			The function \(s^\beta\) is understood to be the main brunch if \(\beta\) is not an integer.
			Similarly for the function \(\log s\) showing up below.
		\end{enumerate}		
	\end{remark}
	\begin{theorem}\label{Theroem IV}\cite[Th\'{e}or\`{e}me IV]{Delange54}
		Let \(\beta\) be any real number, and let \(\gamma\) be an integer\(\geq1\).
		If there exists holomorphic functions \(g_j\) and \(h\) in a neighbourhood of \(\Re(s)\geq\alpha\), and \(g_q(\alpha)\neq0\), such that
		\begin{equation*}
			f(s)=(s-\alpha)^{-\beta}\sum_{j=0}^{\gamma}g_j(s)\Bigl(\log\frac{1}{s-\alpha}\Bigr)^j+h(s),
		\end{equation*}
		for all \(\Re(s)>\alpha\), then 
		\begin{enumerate}
			\item when \(\beta\) is not an integer\(\leq0\), we have as \(t\to\infty\) that
			\begin{equation*}
				A(t)\sim\frac{g_\gamma(\alpha)}{\Gamma(\beta)}e^{\alpha t}t^{\beta-1}(\log t)^{\gamma};
			\end{equation*}
			\item when \(\beta=-m\), with \(m\) an integer\(\geq0\), we have as \(t\to\infty\) that
			\begin{equation*}
				A(t)\sim(-1)^m m!\gamma g_{\gamma}(\alpha)e^{\alpha t}t^{-m-1}(\log t)^{\gamma-1}.
			\end{equation*} 
		\end{enumerate}
	\end{theorem}
	
	\section{Class Field Theory}\label{section:CLF}
	
	For the convenience of the reader, we write down here the main results used from Class Field Theory.
	Let \(K\) be any number field.
	Denote by \(\mathbb{A}_K=\prod_{\mathfrak{p}}'K_{\mathfrak{p}}\) the \emph{ring of ad{\`e}les} of \(K\), and by \(J_K:=\mathbb{A}_K^*\) the \emph{id{\`e}le group} of \(K\).
	There exists a diagonal embedding \(K^*\to J_K\), and we call elements of \(K^*\) the \emph{principal id{\`e}les}.
	The quotient group \(\operatorname{C}_K=J_K/K^*\) is the \emph{id{\`e}le class group} of \(K\).
	There is a map \((\alpha_{\mathfrak{p}})_{\mathfrak{p}}\mapsto\prod_{\mathfrak{p}\nmid\infty}\mathfrak{p}^{v_{\mathfrak{p}}(\alpha_{\mathfrak{p}})}\), whose kernel is \(J_K^{S_\infty}:=\prod_{v\mid\infty}K_v^*\times\prod_{v\nmid\infty}\mathscr{O}_v^*\).
	So, we get the following short exact sequence
	\begin{equation}\label{eqn:CFT 1}
		1\to J_K^{S_\infty}K^*/K^*\to\operatorname{C}_K\to\operatorname{Cl}_K\to1
	\end{equation}
	More generally, for any finite set \(S\) of primes containing infinity, we obtain
	\begin{equation}\label{eqn:CFT 2}
		1\to J_K^SK^*/K^*\to\operatorname{C}_K\to\operatorname{Cl}^S_K\to1.
	\end{equation}
	by the similar idea.
	Using the above short exact sequence, we have the following result.
	\begin{lemma}
		If \(S\) is a finite set of primes containing infinity such that \(\operatorname{Cl}^S_K=1\), then \(J_K^SK^*/K^*\cong\operatorname{C}_K\).
	\end{lemma}
	Since \(J_K^S\) in general is easier to work with, let's consider the following short exact sequence.
	\begin{equation}\label{eqn:CFT 3}
		1\to K^S\to J_K^S\to J_K^SK^*/K^*\to1,
	\end{equation}
	where \(K^S\) is the \(S\)-units of \(K\), which could be defined as \(K^*\cap J_K^S\).
	Since \(J_K^SK^*/K^*=J_K^S/J_K^S\cap K^*\), the notation itself justifies the sequence.
	When \(K=\mathbb{Q}\), we take \(S=S_\infty\), and the short exact sequence says that
	\begin{equation*}
		1\to\{\pm1\}\to J_{\mathbb{Q}}^{S_\infty}\to J_{\mathbb{Q}}^{S_\infty}\mathbb{Q}^*/\mathbb{Q}^*\to1.
	\end{equation*}
	If \(G\) is any finite abelian group, we have
	\begin{equation*}
		1\to\operatorname{Hom}(J_{\mathbb{Q}}^{S_\infty}\mathbb{Q}^*/\mathbb{Q}^*,G)\to\operatorname{Hom}(J_{\mathbb{Q}}^{S_\infty},G)\to\operatorname{Hom}(\{\pm1\},G)\to\cdots
	\end{equation*}
	Since for any map \(\rho:J_{\mathbb{Q}}^{S_\infty}\to G\), either \(\rho(-1)=1\) or \(\rho(-1)\neq1\), we see that a homomorphism \(\rho:\prod_{p}\mathbb{Z}_p^*\to G\) could always be extended to a homomorphism \(\tilde{\rho}:J_{\mathbb{Q}}^{S_\infty}\to G\) such that it is an image of an element of \(\operatorname{Hom}(J_{\mathbb{Q}}^{S_\infty}\mathbb{Q}^*/\mathbb{Q}^*,G)\), by setting up a suitable value of \(\rho_{\infty}(-1)\).
	Clearly such a value is unique.
	In other words, there exists a unique map \(\tilde{\rho}:J_{\mathbb{Q}}^{S_\infty}\mathbb{Q}^*/\mathbb{Q}^*\to G\) such that \(\tilde{\rho}\vert_p=\rho\vert_p\) for all \(p\nmid\infty\).
	So, we have the following result.
	\begin{lemma}\label{lemma:CFT of Q}
		Let \(G\) be an abelian group.
		There is a one-to-one correspondence between \(\operatorname{Hom}(\operatorname{C}_{\mathbb{Q}},G)\) and \(\operatorname{Hom}(\prod_{p\nmid\infty}\mathbb{Z}_p^*,G)\), given by \(\tilde{\rho}\mapsto\prod_{p\nmid\infty}\tilde{\rho}\vert_p\).
		In particular, \(\operatorname{Hom}(\prod_p\mathbb{Z}_p^*,G)\) corresponds to the set of Artin reciprocity maps, i.e., abelian \(G\)-fields, under this map.
	\end{lemma}
	\begin{proof}
		The only thing we need to prove here is that if \(\tilde{\rho}\) is an Artin reciprocity map, then \(\rho:=\prod_{p\nmid\infty}\tilde{\rho}\vert_p\) is already surjective.
		Since \(\tilde{\rho}\) corresponds to a class field \(L\), we know that \(G\cong\operatorname{Gal}(L/\mathbb{Q})\) must be generated by all the inertia subgroups \(I(p)\) where \(p\) runs over all rational primes.
		Using Global-to-Local principal, we know that \(I(p)\) is just given by \(\rho(\mathbb{Z}_p^*)\).
		Therefore, \(\operatorname{im}(\rho)=G\).
		So the image of the set of Artin reciprocity maps under the correspondence \(\tilde{\rho}\mapsto\prod_{p\mid\infty}\tilde{\rho}\vert_p\) is just \(\operatorname{Sur}(\prod_{p\nmid\infty}\mathbb{Z}_p^*,G)\).
	\end{proof}
	\begin{lemma}
		Let \(K\) be a fixed Galois number field with an isomorphism \(\operatorname{Gal}(K/\mathbb{Q})\cong H_1\), and let \(H_2\) be a finite abelian group.
		If \(S\) is a finite set of primes of \(K\) containing the ones at infinity such that \(\mathscr{O}_{K}^S\) has class number \(1\), then there is a one-to-one correspondence between \(\operatorname{Hom}_{H_2}(J_{K}^SK^*/K^*,H_2)\) and \(\operatorname{Hom}_{H_1}(\operatorname{C}_{K},H_2)\), where \(J_K\) is the group of id{\`e}les, and \(\operatorname{C}_K\) is the id{\`e}les class group, and \(H_2\) is equipped with a fixed \(H_1\)-action.
	\end{lemma}
	\begin{proof}
		The correspondence is induced by the isomorphism \(J_K^SK^*/K^*\cong\operatorname{C}_K\).
		Note that since \(\operatorname{C}_K\) is an \(H_1\)-module, we see that \(J_K^SK^*/K^*\) is also a well-defined \(H_1\)-module.
	\end{proof}
	\begin{lemma}
		Let \(K\) be a fixed Galois number field with an isomorphism \(\operatorname{Gal}(K/\mathbb{Q})\cong H_1\), and let \(H_2\) be a finite abelian group, and let \(L/K\) be an \(H_2\)-extension.
		The field extension \(L/\mathbb{Q}\) is Galois with \(G:=\operatorname{Gal}(L/\mathbb{Q})\) if and only if the corresponding Artin reciprocity map \(\rho:\operatorname{C}_{K}\to H_2\) is an \(H_1\)-morphism, in which case \(H_2\) is equipped with the \(H_1\)-action satisfying the short exact sequence
		\begin{equation}\label{eqn:galois action on idele}
			1\to H_2\to G\to H_1\to1.
		\end{equation}
	\end{lemma}
	\begin{proof}
		Let \(L/K\) be an \(H_2\)-extension, and let \(M\) be the Galois closure of \(L/\mathbb{Q}\).
		We treat \(\operatorname{C}_{K},\operatorname{C}_{L}\) as subgroups of \(\operatorname{C}_{M}\).
		
		First assume that \(L/\mathbb{Q}\) is a Galois extension, which means that we already have the short exact sequence (\ref{eqn:galois action on idele}).
		And the action of \(H_1\) on \(H_2\) is exactly the action of conjugation.
		By Class Field Theory we know that \(\operatorname{C}_{K}/N\operatorname{C}_{L}\cong\operatorname{Gal}(L/K)\cong H_2\), where \(N\) means norm \(\operatorname{Nm}_{L/K}\).
		Since \(L\) is Galois, we see that \(N\operatorname{C}_{L}\) is a well-defined \(H_1\)-module, which means that the Artin reciprocity map \(\rho:\operatorname{C}_{K}\to H_2\) is an \(H_1\)-morphism.
		All the \(H_1\)-actions showing up are consistent with the short exact sequence (\ref{eqn:galois action on idele}), which gives the Galois action.
		
		Conversely, let \(\rho:\operatorname{C}_{K}\to H_2\) be the Artin reciprocity map corresponding to \(L/K\).
		If \(L'\neq L\) is a conjugate, then we see that \(N\operatorname{C}_{L}\neq N\operatorname{C}_{L'}\), because they corresponds to different abelian extensions of \(K\).
		In particular, if \(g\in\operatorname{Gal}(M/\mathbb{Q})\) is an element such that \(g L=L'\), then we see that \(g\vert_{K}\) induces the map \(N\operatorname{C}_{L}\to N\operatorname{C}_{L'}\).
		But this already shows that \(\rho\) cannot be an \(H_1\)-morphism.
		And we are done.
	\end{proof}
	In general, we obtain the following results.
	\begin{proposition}\label{prop:CFT obstruction}
		Let \(G\) be a finite abelian group.
		If \(\rho:J_K^{S_\infty}\to G\) be a surjective homomorphism, then there exists at most \(\lvert\operatorname{Hom}(\operatorname{Cl}_K,G)\rvert\) many different Artin reciprocity maps \(\chi:\operatorname{C}_K\to G\) such that \(\chi\vert_{\mathfrak{p}}=\rho\vert_{\mathfrak{p}}\) for all \(\mathfrak{p}\nmid\infty\).
	\end{proposition}
	\begin{proof}
		Using (\ref{eqn:CFT 1}), we see that
		\begin{equation*}
			1\to\operatorname{Hom}(\operatorname{Cl}_K,G)\to\operatorname{Hom}(\operatorname{C}_K,G)\to\operatorname{Hom}(J_K^{S_\infty}K^*/K^*,G)\to\cdots.
		\end{equation*}
		If a homomorphism \(\psi:J_K^{S_\infty}K^*/K^*\) could be lifted to \(\operatorname{Hom}(\operatorname{C}_K,G)\), then size of the preimages is determined by \(\lvert\operatorname{Hom}(\operatorname{Cl}_K,G)\rvert\) according to this long exact sequence.
		Similarly, by (\ref{eqn:CFT 3}), we have
		\begin{equation*}
			1\to\operatorname{Hom}(J_K^{S_\infty}K^*/K^*,G)\to\operatorname{Hom}(J_K^{S_\infty},G)\to\operatorname{Hom}(K^S,G)\to\cdots
		\end{equation*}
		The long exact sequence says that if \(\rho:J_K^{S_\infty}\to G\) is a homomorphism with a lift in \(\operatorname{Hom}(J_K^{S_\infty}K^*/K^*,G)\), then the preimage is unique.
		Combining the two parts, we obtain the result immediately.
	\end{proof}
	Finally, we cite another result from Neukirch here.
	\begin{proposition}\cite[VI.1.4]{neukirch2013algebraic}\label{prop:CFT S-ideles}
		If \(S\) is a sufficiently big finite set of places of \(K\), then \(J_K=J_K^SK^*\), i.e., \(\operatorname{C}_K=J_K^SK^*/K^*\).
	\end{proposition}
\end{appendices}
    
\bibliographystyle{plain}
\bibliography{references}

\end{document}